\documentclass[12pt]{amsart}
\usepackage{amsthm}
\usepackage[T1]{fontenc}
\usepackage{lmodern}
\usepackage{amsmath,amssymb,mathtools,array,booktabs,longtable}
\usepackage{tikz}
\usetikzlibrary{arrows.meta,positioning,shapes.geometric,fit,calc,
decorations.pathreplacing}
\usepackage[hidelinks]{hyperref}
\definecolor{cutblue}{RGB}{32,92,145}
\definecolor{cutred}{RGB}{178,70,56}
\definecolor{cutgreen}{RGB}{28,117,93}
\definecolor{cutpurple}{RGB}{113,78,150}
\tikzset{
 value diagram/.style={x=1cm,y=1cm,every node/.style={font=\small}},
 value axis/.style={-{Latex[length=1.6mm]},black!25,line width=.45pt},
 value ray/.style={-{Latex[length=2mm]},line width=1.7pt},
 cutguide/.style={black!20,densely dashed,line width=.5pt},
 shift arrow/.style={-{Latex[length=1.8mm]},densely dashed,line width=.8pt},
 open cut/.style={circle,draw=black!65,fill=white,line width=.8pt,
                  minimum size=4.5pt,inner sep=0pt},
 closed cut/.style={circle,draw=black!65,fill=black!65,
                    minimum size=4.5pt,inner sep=0pt},
 row label/.style={anchor=west,font=\small\bfseries},
 formula label/.style={font=\small},
 panel title/.style={anchor=west,font=\small\bfseries}
}

\newtheorem{theorem}{Theorem}[section]
\newtheorem{proposition}[theorem]{Proposition}
\newtheorem{lemma}[theorem]{Lemma}
\newtheorem{corollary}[theorem]{Corollary}

\theoremstyle{definition}

\newtheorem{example}[theorem]{Example}
\theoremstyle{remark}
\newtheorem{remark}[theorem]{Remark}

\newcommand{\OK}{\mathcal O_K}
\newcommand{\OL}{\mathcal O_L}
\newcommand{\mL}{\mathfrak m_L}
\newcommand{\Tr}{\operatorname{Tr}}
\newcommand{\Nm}{\operatorname{N}}
\newcommand{\ann}{\operatorname{ann}}
\newcommand{\Gal}{\operatorname{Gal}}
\newcommand{\Dzero}{\mathfrak D_{\mathrm{prim}}}
\newcommand{\DD}{\mathfrak D}
\newcommand{\CC}{\mathfrak C}

\title{Ramification invariants of pure extensions}
\author{Josnei Novacoski}
\address{Departamento de Matem\'{a}tica, Instituto de Ci\^encias
Matem\'{a}ticas e de Computa\c{c}\~ao, Universidade de S\~ao Paulo,
Avenida Trabalhador S\~ao-carlense, 400, CEP 13566--590,
S\~ao Carlos--SP, Brazil}
\email{josnei@icmc.usp.br}
\subjclass[2020]{Primary 13A18; Secondary 12J10, 11S15}
\keywords{Pure extensions, valuation rings, defect, trace ideals,
differents, ramification ideals, Swan conductors}
\date{}
\thanks{During the realization of this project the author was supported by a grant from Funda\c{c}\~ao de Amparo \`a Pesquisa do Estado de S\~ao Paulo (process number 2025/10540-0) and a grant from Conselho Nacional de Desenvolvimento Cient\'ifico e Tecnol\'ogico (process number 317185/2025-0).}
\begin{document}

\begin{abstract}
We study ramification invariants of finite separable extensions of
valued fields that are pure on a primitive element.  In the immediate
unibranched case, we prove that the extended trace ideal is generated
by minimal-polynomial derivatives of integral primitive elements and
that the Dedekind different is its divisorial closure.  For Galois
extensions, we express this derivative ideal as a product of subgroup
ramification ideals.  These formulas extend prime-degree calculations
to pure extensions of higher degree.  We compute the differential
module, traces, differents, ramification ideals, and character Swan
ideals of a pure immediate elementary abelian extension of degree $p^2$.  A totally ramified example of the same degree gives
a comparison with the defectless case.
\end{abstract}

\maketitle

\section{Introduction}

Let $(L/K,v)$ be a finite separable extension of valued fields of degree
$n>1$, and
let $\OK$ and $\OL$ be the corresponding valuation rings.
We study how purity allows descriptions of ramification invariants
in prime degree to extend to higher degree.  Our starting point is the
computation of K\"ahler differentials by Novacoski and Spivakovsky
\cite{NovacoskiSpivakovsky} and the description of ramification ideals
arising from the work of Dutta and Novacoski \cite{DuttaNovacoski}.
Both express ramification in terms of linear approximations to a
primitive element.  We show that these approximations also determine
the trace ideal and, through trace duality, the Dedekind different.
In the Galois case, the resulting product formula incorporates every
conjugate-distance level, extending the single-level formula in prime
degree.  The examples show why retaining the full ideals, including
their boundary behavior, matters in extensions with defect.

Over complete discretely valued fields, the different and trace are
related to the classical ramification groups.  In the totally ramified
Galois case, let $d$ be the integral different exponent, let
$G_i$ be the lower ramification groups, and put $e=[L:K]$.  Then
\[
 d=\sum_{i\geq0}(|G_i|-1),\qquad
 \Tr_{L/K}(\OL)=\mathfrak m_K^{\lfloor d/e\rfloor}
\]
by the classical formulas in \cite{SerreLocal}.  In this setting the
valuation ring is monogenic over the base valuation ring, and the
different is generated by a minimal-polynomial derivative and
annihilates $\Omega_{\OL/\OK}$.

For more general valuations, defect and nonprincipal ideals require
additional care.  For a finite unibranched extension, write
\[
 d_{L/K}=\frac{[L:K]}{e(L/K)f(L/K)}.
\]
Here $e(L/K)$ and $f(L/K)$ are the value-group index and residue-field
degree, respectively.  If $p=\operatorname{char}Kv>0$, then $d_{L/K}=p^s$ for an integer
$s\geq0$, called the defect of $(L/K,v)$.  An immediate extension has
$e=f=1$, so a nontrivial immediate extension has nontrivial defect.
In this situation $\OL$ is not a finite $\OK$-module.  Moreover,
a value cut may have no least element, and its infimum need not
determine the corresponding ideal.

Several approaches extend ramification theory beyond the classical
setting.  Kato and Abbes--Saito extend conductor theory
to imperfect residue fields \cite{KatoSwan,AbbesSaito}.
Thatte computes prime-degree ramification and conductor ideals,
including defect extensions \cite{ThatteAS,ThatteMixed}.
Kato--Thatte define an upper filtration indexed by ideals for arbitrary
henselian valuation rings \cite{KatoThatte}.
The ramification ideals, differential modules, and trace ideals studied
by Kuhlmann and collaborators provide further descriptions in
arbitrary rank \cite{KuhlmannRzepka,CKR,KuhlmannTopics}.

To describe the role of purity, fix a primitive element $\theta$ of
$L/K$, put $n=[L:K]$, and let $g$ be its monic minimal polynomial.  We say that the extension is pure on $\theta$
if the truncations defined by linear key polynomials $X-c$, $c\in K$,
recover the values of all polynomials of degree less than $[L:K]$.
For such a generator, set
\[
 D=v(\theta-K)=\{v(\theta-c):c\in K\}.
\]
The two descriptions that motivate our results are expressed in terms
of $D$.  To state them, suppose that the extension
is immediate and unibranched: its value group and residue field agree
with those of $K$, and $v$ has a unique extension from $K$ to $L$.
For a final segment $S$ of $vL$, write
\[
 I_S=\{0\}\cup\{x\in L^\times:vx\in S\}.
\]
Here $\operatorname{Final}(E)$ denotes the smallest final segment
containing a set $E\subseteq vL$.  Since $D$ is an initial segment in
the immediate case, $-D$ is already a final segment.  Set
\[
 \alpha=-D,\qquad
 \beta=\operatorname{Final}\bigl(v(g'(\theta))-nD\bigr),
\]
where $nD=\{nd:d\in D\}$.  The computation of Novacoski and
Spivakovsky \cite{NovacoskiSpivakovsky} then gives the description
of the module of K\"ahler differentials
\begin{equation}\label{eq:intro-differentials}
 \Omega_{\OL/\OK}\simeq I_\alpha/I_\beta.
\end{equation}

For the second description assume, in addition, that $L/K$ is Galois
and put $G=\Gal(L/K)$.  For a nontrivial subgroup $H\leq G$, its
ramification ideal and its least conjugate-distance value are
\[
 \begin{aligned}
 I_H&=\left(\frac{\sigma x-x}{x}:
             x\in L^\times,\ \sigma\in H\right)\OL,\\
 \beta_H&=\min_{\sigma\in H\setminus\{1\}}v(\sigma\theta-\theta).
 \end{aligned}
\]
The work of Dutta and Novacoski \cite{DuttaNovacoski} relates purity
and distances between conjugates to ramification ideals.  The resulting
subgroup formula, stated in \cite{NovacoskiRam}, is
\begin{equation}\label{eq:intro-ramification}
 I_H=I_{\beta_H-D}.
\end{equation}
Formulas \eqref{eq:intro-differentials} and
\eqref{eq:intro-ramification} show how purity relates these two
invariants to the approximation set $D$.  The results below extend
this description to the trace ideal and the different.

We first consider the trace ideal and the different.  For an $\OL$-submodule
$J\subseteq L$, write $\OL:J=\{x\in L:xJ\subseteq\OL\}$.  Set
\[
 \begin{aligned}
 T&=\Tr_{L/K}(\OL)\OL,\\
 \Dzero&=(h_b'(b):b\in\OL,\ K(b)=L)\OL,\\
 \CC&=\{x\in L:\Tr_{L/K}(x\OL)\subseteq\OK\},\\
 \DD&=\OL:\CC,
 \end{aligned}
\]
Here $h_b$ is the monic minimal polynomial of $b$ over $K$;
$\CC$ is the complementary module and $\DD$ the Dedekind different.
We call $\Dzero$ the primitive-generator different and write
$J^v=\OL:(\OL:J)$ for the divisorial closure of a fractional ideal
$J$.  Our first main result identifies these ideals in the pure
immediate case.

\begin{theorem}\label{thm:intro-trace}
Assume that $(L/K,v)$ is immediate, unibranched, and pure on $\theta$.
Then
\[
 T=\Dzero,\qquad \CC=\OL:\Dzero,\qquad \DD=\Dzero^v.
\]
\end{theorem}

The principal assertion of Theorem~\ref{thm:intro-trace} is
$T=\Dzero$: purity identifies the trace ideal with the ideal generated
by minimal-polynomial derivatives of integral primitive elements.
The complementary-module and Dedekind-different identities are then
consequences of trace duality.  Indeed, $\CC=\OL:T$ and $\DD=T^v$
are already known under the more general hypothesis $vL=vK$
\cite[Theorem 5.4]{CKR}.  Thus the theorem supplies an explicit
approximation-theoretic description of the ideal entering those
identities.
The proof, given in Theorem~\ref{thm:pure-trace-different}, uses linear
truncations and the interpolation formula for the trace.  If $g'$ is constant, it also gives explicit trace
generators (Corollary~\ref{thm:constant-derivative}).

In the Galois case, the factorization of $g'(\theta)$ into conjugate
differences leads to a corresponding factorization of the
primitive-generator different.  This connects the trace calculation
to the subgroup ramification ideals of
\eqref{eq:intro-ramification}.

\begin{theorem}\label{thm:intro-product}
Assume that $(L/K,v)$ is immediate, unibranched, Galois, and pure
on $\theta$.  Then
\[
 \Dzero=\prod_{\sigma\in G\setminus\{1\}}I_{\langle\sigma\rangle}.
\]
\end{theorem}

This is proved in Theorem~\ref{thm:profile-factorization}.  In degree
$p$, all nonidentity conjugate differences have the same value, so the
product reduces to $I_G^{p-1}$.  In higher degree, distinct values may
occur, and each associated ramification ideal appears with the
multiplicity of that value.
Corollary~\ref{cor:cut-stabilizer} characterizes when the product
still equals $I_G^{[L:K]-1}$.

The example in \cite[Section 5.2]{NovacoskiDistances} provides a
concrete application beyond prime degree: it has depth one, degree
$p^2$, and defect $p^2$.  Its purity allows the preceding formulas to
be applied despite the nontrivial defect.  With the notation of
Section~\ref{sec:examples}, we obtain
\[
 \begin{aligned}
 T=\Dzero&=t^{p-1}\mL,\\
 \DD=\ann\Omega_{\OL/\OK}&=t^{p-1}\OL.
 \end{aligned}
\]
By Proposition~\ref{prop:ordinary-log-equal-values}, the same quotient
also describes $\Omega^{\log}_{\OL/\OK}$.
The two subgroup ramification ideals are $\mL$ and $t\mL$.
Put $A=\mathcal O_{K^h}$, where $K^h$ is the henselization of $K$.
The degree-$p$ character quotients of $LK^h/K^h$ have Swan ideals $\mathfrak m_A$ and $t^p\mathfrak m_A$.
These conductor computations apply the existing prime-degree theory
of Thatte and the ideal-indexed upper filtration of Kato--Thatte
\cite{ThatteAS,KatoThatte}; they do not require a higher-degree
parameter theory.  The character conductors determine the logarithmic
upper filtration.
The norm of $t\mL$ in the full extension is instead
$t^{p^2}\mathfrak m_A$, showing that this norm need not be a
character conductor.  We also compute a totally ramified degree-$p^2$
example to compare these formulas with the defectless case.

Table~\ref{tab:invariant-comparison} summarizes these formulas and
their scope.
Except for the conductor rows, the three cases concern finite separable
immediate unibranched Galois extensions.  The degree-$p$ case has
$p=\operatorname{char}Kv$; the pure case has degree $n$ and assumes
purity on $\theta$; the general case drops that assumption.
Retain the notation $D$, $\alpha$, $\beta$, and $T$ introduced above,
and put
\[
 \begin{aligned}
 b_\sigma&=v(\sigma\theta-\theta),&
 b_H&=\min_{\sigma\in H\setminus\{1\}}b_\sigma,\\
 J&=\prod_{\sigma\ne1}I_{b_\sigma-D}.
 \end{aligned}
\]
All ideal colons in the table are taken in $L$.
For the conductor rows assume a henselian rank-one base of characteristic
$p$.  For an order-$p$ character $\chi$, let $\mathcal H_\chi$
be its Swan ideal and put $Q_\chi=L^{\ker\chi}$.  If $\mathfrak a$
is a nonzero proper ideal in the valuation ring of a separable closure,
write
\[
 \begin{aligned}
 N_\chi&=(\Nm_{Q_\chi/K}I_{\Gal(Q_\chi/K)})\OK,\\
 B_{\mathfrak a}&=\{\chi\ne0:
          \mathfrak a\cap\OK\subsetneq\mathcal H_\chi\},\\
 K_{\mathfrak a}&=\bigcap_{\chi\in B_{\mathfrak a}}\ker\chi.
 \end{aligned}
\]
Thus $B_{\mathfrak a}$ consists of the nonzero characters that vanish
on the upper ramification image at index $\mathfrak a$.  Intersecting
their kernels recovers that image; an empty intersection means $G$.
The strict inclusion retains the behavior at the conductor boundary.
In the last row, $U_{\mathfrak a}$ denotes the image in $G$ of the
logarithmic upper ramification group indexed by $\mathfrak a$.
This row assumes $G$ elementary abelian, and $C_{\mathfrak a}$ abbreviates
$\mathfrak a\cap\OK\supseteq\mathcal H_\chi$ in degree $p$.
A dash indicates that no formula in approximation data is asserted
here; it is not a claim that the invariant is unknown in the literature.

\begin{table}[tbp]
\centering
\small
\setlength{\tabcolsep}{3pt}
\renewcommand{\arraystretch}{1.2}
\caption{Invariant formulas and their scope.}
\label{tab:invariant-comparison}
\begin{tabular}{@{}>{\raggedright\arraybackslash}p{\dimexpr.26\textwidth-4.5pt\relax}
 >{\centering\arraybackslash}p{\dimexpr.25\textwidth-4.5pt\relax}
 >{\centering\arraybackslash}p{\dimexpr.27\textwidth-4.5pt\relax}
 >{\centering\arraybackslash}p{\dimexpr.22\textwidth-4.5pt\relax}@{}}
\toprule
Invariant & Degree $p$ & Pure case & General case\\
\midrule
K\"ahler differentials $\Omega$ & \mbox{$I_G/I_G^p$} & \mbox{$I_\alpha/I_\beta$} & ---\\
\addlinespace
Annihilator $\ann\Omega$ &
\mbox{$\OL\cap(I_G^p:I_G)$} &
\mbox{$\OL\cap(I_\beta:I_\alpha)$} & ---\\
\addlinespace
Primitive different $\Dzero$ & $I_G^{p-1}$ & $J$ & ---\\
\addlinespace
Trace ideal $T$ & $I_G^{p-1}$ & $J$ & ---\\
\addlinespace
Complementary module $\CC$ & $\OL:I_G^{p-1}$ & $\OL:J$ & $\OL:T$\\
\addlinespace
Dedekind different $\DD$ & $(I_G^{p-1})^v$ & $J^v$ & $T^v$\\
\addlinespace
Ramification ideal $I_H$ & $I_{b_G-D}$ & $I_{b_H-D}$ & ---\\
\addlinespace
Norm value $v\Nm(x)$ & $p\,vx$ & $n\,vx$ & $n\,vx$\\
\addlinespace
Swan ideal $\mathcal H_\chi$ & $N_\chi$ & $N_\chi$ & $N_\chi$\\
\addlinespace
Upper image $U_{\mathfrak a}$ &
$G$ if $C_{\mathfrak a}$;\newline $\{1\}$ otherwise &
$K_{\mathfrak a}$ & $K_{\mathfrak a}$\\
\bottomrule
\end{tabular}
\end{table}

The pure differential formula comes from
\cite{NovacoskiSpivakovsky}, and the prime-degree description is
included in \cite{CK,CKR}.
The primitive different and trace formulas are proved in
Theorems~\ref{thm:profile-factorization} and~\ref{thm:pure-trace-different}.
The general trace-duality entries are \cite[Theorem 5.4]{CKR};
the norm and character formulas are
Propositions~\ref{prop:norm} and~\ref{prop:character-reconstruction},
using \cite{ThatteAS,KatoThatte} for the conductor identification.

The paper is organized as follows.  Section~\ref{sec:preliminaries}
fixes notation and proves the affine-presentation criterion.
Sections~\ref{sec:differentials-vs-ramification}--\ref{sec:trace-profiles}
treat differentials, ramification ideals, differents, norms, and traces.
Section~\ref{sec:swan} recalls the conductor results used in
Section~\ref{sec:examples}, where the two examples are computed.

\section{Preliminaries}
\label{sec:preliminaries}
Write valuations additively, $v0=\infty$, and put
$\OK=\{x\in K:vx\geq0\}$, $\mathfrak m_K=\{x:vx>0\}$.
The residue field is $Kv=\OK/\mathfrak m_K$.
An extension is immediate if $vL=vK$ and $Lv=Kv$, and unibranched
if the valuation of $K$ has a unique prolongation to $L$.
Henselian bases give unibranched algebraic extensions.
For a value final segment $S$, write
$I_S=\{0\}\cup\{x\in L^\times:vx\in S\}$.
Expressions defining $S$ always include final-segment generation.
For fractional ideals put
\[
 \OL:J=\{x\in L:xJ\subseteq\OL\},\qquad
 J^v=\OL:(\OL:J).
\]
In dense rank one, a boundary in the value group becomes closed under
this divisorial operation.

Let $L=K(\theta)$ be separable of degree $n>1$, with monic minimal
polynomial $g$.  Put $D_1=v(\theta-K)$ and
\[
 \begin{aligned}
 P&=\sum_jP_{j,c}(X-c)^j,\\
 \nu_c(P)&=\min_j\{vP_{j,c}+jv(\theta-c)\}.
 \end{aligned}
\]
Purity on $\theta$ means that every $P$ with $\deg P<n$ admits
$c\in K$ such that $\nu_c(P)=vP(\theta)$.
If $v(\theta-d)\geq v(\theta-c)$, expanding $X-c=(X-d)+(d-c)$
gives $\nu_d(P)\geq\nu_c(P)$.  Thus finitely many polynomial
values can be recovered at once.  In the immediate case $D_1$ has
no maximum; choose $t_c\in K$ with $vt_c=-v(\theta-c)$ and set
$u_c=t_c(\theta-c)$.
\begin{proposition}
\label{rem:CKR-purity}
Let $(L/K,v)$ be finite, separable, immediate, and unibranched, and
let $L=K(\theta)$.  Purity on $\theta$ is equivalent to the existence
of an exhaustive nested presentation
\[
 \begin{gathered}
 \OL=\bigcup_\alpha\OK[b_\alpha],\\
 b_\alpha=r_\alpha(\theta-c_\alpha)\in\OL,
 \quad r_\alpha\in K^\times,\quad c_\alpha\in K.
 \end{gathered}
\]
For a pure extension, the normalized linear generators satisfy all
the presentation hypotheses of \cite[Theorem 1.1]{CKR}, including
compatible scaling and divisibility.  Conversely, such an affine
presentation implies purity whenever $vL=vK$, without requiring
immediacy.
\end{proposition}

\begin{proof}
Assume first that the extension is pure.  For $y=P(\theta)\in\OL$,
$\deg P<n$, choose $c$ with $\nu_c(P)=vy$.  If
$P=\sum_i P_{i,c}(X-c)^i$, then
$v(P_{i,c}t_c^{-i})\geq vy\geq0$, so
$y=\sum_i(P_{i,c}t_c^{-i})u_c^i\in\OK[u_c]$.
This proves exhaustion.  Order a cofinal family by $\delta_c=v(\theta-c)$ and write
$u_c=t_c(\theta-c)$.  If $\delta_c\leq\delta_d$, then
\[
 \begin{gathered}
u_c=\frac{t_c}{t_d}u_d+t_c(d-c),\\
 v(t_c/t_d)=\delta_d-\delta_c\geq0,\\
 v(t_c(d-c))\geq0.
 \end{gathered}
\]
Thus both transition coefficients belong to $\OK$, and
$\OK[u_c]\subseteq\OK[u_d]$.  To meet the theorem's additional
requirement that its scaling elements lie in $\OK$, choose a common
$k\in K^\times$ with $v(kt_c)\geq0$ and set $a_c=kt_c$.
Such a common bound exists: unibranchedness and separability give
$\delta_c\leq v(\theta-\tau\theta)$ for any fixed nonidentity
$K$-embedding $\tau$, and the value group of a finite algebraic extension is cofinal over
$vK$.  Choose $vk$ at least this bound.  The ratios $a_c/a_d=t_c/t_d$
are unchanged.  Together with $\OL=\bigcup_c\OK[u_c]$, this verifies
all the presentation requirements of that theorem.

Conversely, suppose $vL=vK$ and the affine presentation exists.
Take a nonzero polynomial
$f\in K[X]$ of degree less than $n$, choose $a\in K^\times$ with
$va=v(f(\theta))$, and place $f(\theta)/a$ in some $\OK[b_\alpha]$.
Since $b_\alpha$ is integral and has degree $n$, its unique expansion
of degree less than $n$ has integral coefficients.  Substitution gives
$\nu_{c_\alpha}(f)\geq va$; the opposite inequality always holds,
so equality follows.  Finite collections are handled in a common
larger algebra.

The compatible affine maps in \cite[Theorem 1.1]{CKR} put every
 generator on a cofinal tail in the form $r_\alpha(\theta-c_\alpha)$
after choosing one of them as $\theta$.  Consequently, in this
 equal-value-group setting, existence of such a presentation expresses
purity on a suitable primitive generator.  It does not assert purity
on every primitive generator, nor does this argument remove the
hypothesis $vL=vK$ for the other pure classes.
\end{proof}

Indeed, \cite[Theorem 1.1]{CKR} gives
$\Omega_{\OL/\OK}\simeq U/UV$, where
\[
 \begin{aligned}
 U&=(kt_c:c\in K)\OL,\\
 V&=(h_{u_c}'(u_c):c\in K)\OL.
 \end{aligned}
\]
Here $h_{u_c}$ is the minimal polynomial of $u_c$.
Writing $I_\alpha=(t_c:c\in K)\OL$ and cancelling the common
factor $k$ yields
\[
 \Omega_{\OL/\OK}\simeq I_\alpha/(I_\alpha V).
\]
In the pure immediate case the derivative calculation below gives
$V=\Dzero$, and $v(I_\alpha V)$ is generated by
$v(g'(\theta))-nD_1(\theta,K)$.  Thus $I_\alpha V=I_\beta$,
recovering the corresponding Novacoski--Spivakovsky quotient.
The proposition supplies the presentation needed for this alternative
derivation; it does not assert that all cases treated in
\cite{NovacoskiSpivakovsky} follow from the present immediate argument.

Expressions such as $-D_1$ and $a-D_1$ denote the final segments
generated by the negatives or translated negatives of the approximation
values.  Likewise, $mD_1$ denotes the set of scalar multiples $md$,
$d\in D_1$, with final-segment generation understood when defining an
ideal.  This convention also applies when $vL\ne vK$.

The relative differential module $\Omega_{\OL/\OK}$ is generated by
$dx$, subject to additivity, the Leibniz rule, and $da=0$ for $a\in\OK$.
We use
\[
 \begin{aligned}
 T&=\Tr_{L/K}(\OL)\OL,\\
 \CC&=\{x\in L:\Tr_{L/K}(x\OL)\subseteq\OK\},\\
 \DD&=\OL:\CC,\\
 \Dzero&=(h_b'(b):b\in\OL,\ K(b)=L)\OL.
 \end{aligned}
\]
Here $h_b$ is the monic minimal polynomial.  The all-elements ideal
$\mathfrak D_{\mathrm{all}}=(h_b'(b):b\in\OL\setminus K)\OL$
is distinct in composite degree; see \cite[Section 4.3]{KuhlmannRzepkaTopics}.
For a valuation-invariant Galois extension put
\[
 \begin{aligned}
 \mathcal J_\sigma&=((\sigma x-x)/x:x\in L^\times)\OL,\\
 I_H&=\sum_{\sigma\in H\setminus\{1\}}\mathcal J_\sigma.
 \end{aligned}
\]
Norm ideals mean ideals generated by element norms.
The logarithmic module $\Omega^{\log}_{\OL/\OK}$ additionally has
symbols $d\log x$ for $x\in L^\times$, with
$d\log(xy)=d\log x+d\log y$, $dx=x\,d\log x$ for $x\in\OL\setminus\{0\}$,
and $d\log a=0$ for $a\in K^\times$.
Absolute logarithmic differentials omit these relative base relations.

\section{Differentials and subgroup ramification ideals}
\label{sec:differentials-vs-ramification}

The computations of K\"ahler differentials and ramification ideals are
closely related, but they use the approximation data in different ways.  In
degree $p$ this distinction is largely invisible because the nontrivial
conjugate differences have a common value.  For pure extensions of arbitrary
degree, the distinction becomes essential.

\subsection{K\"ahler differentials}

Throughout this subsection, let $(L/K,v)$ be finite, separable,
immediate, unibranched, and pure on $\theta$.  Put
\[
 D=D_1(\theta,K)=v(\theta-K).
\]
The computation of \cite{NovacoskiSpivakovsky} associates with a complete family of
linear key polynomials two final segments $\alpha$ and $\beta$.  In the pure
case,
\[
 \alpha=-D,
\]
and, for a complete family $\mathcal C$ of linear key polynomials,
\[
 \beta=\operatorname{Final}_{vL}
 \{v(g'(\theta))-\nu_c(g):c\in\mathcal C\}.
\]
Here $\nu_c$ is the linear truncation defined in the preliminaries.
One has
\[
 I_\beta\subseteq I_\alpha,\qquad
 \Omega_{\OL/\OK}\simeq I_\alpha/I_\beta,
\]
and consequently
\[
 \ann\Omega_{\OL/\OK}
 =I_{\{\gamma\in vL:\gamma+\alpha\subseteq\beta\}}\cap\OL.
\]
Thus the module of differentials is represented on the value group by the
part of the numerator final segment $\alpha$ that remains before the
denominator final segment $\beta$ begins.

The segment $\alpha$ is therefore entirely approximation-theoretic, while
$\beta$ contains the algebraic relation supplied by $g$.  The equality
$\Omega_{\OL/\OK}=0$ is equivalent to $\alpha=\beta$.  When the two cuts are
different, the distance between them controls the annihilator of the
quotient.

\begin{figure}[htbp]
\centering
\begin{tikzpicture}[value diagram]
\fill[cutgreen!12] (3.2,-2.45) rectangle (7.2,-1.95);
\draw[cutguide] (3.2,.15) -- (3.2,-2.55);
\draw[cutguide] (7.2,.15) -- (7.2,-2.55);
\foreach \y in {0,-1.1,-2.2}
  \draw[value axis] (2.5,\y) -- (11.5,\y);
\node[row label,text=cutblue] at (0,0) {Numerator};
\node[row label,text=cutred] at (0,-1.1) {Denominator};
\node[row label,text=cutgreen] at (0,-2.2) {Quotient};
\draw[value ray,cutblue] (3.27,0) -- (11.2,0);
\node[open cut,draw=cutblue] at (3.2,0) {};
\draw[value ray,cutred] (7.27,-1.1) -- (11.2,-1.1);
\node[open cut,draw=cutred] at (7.2,-1.1) {};
\draw[cutgreen,line width=2pt] (3.27,-2.2) -- (7.13,-2.2);
\node[formula label,text=cutblue] at (7.4,.34)
  {$\alpha=-D_1(\theta,K)$};
\node[formula label,text=cutred] at (9,-.76) {$\beta=vI_\beta$};
\node[formula label,text=cutgreen] at (5.2,-2.85)
  {values of representatives in $I_\alpha\setminus I_\beta$};
\node[anchor=south] at (3.2,.5) {cut of $\alpha$};
\node[anchor=south] at (7.2,.5) {cut of $\beta$};
\node[anchor=west] at (10.8,-2.65) {$vL$};
\end{tikzpicture}
\caption{The differential quotient $I_\alpha/I_\beta$.
The blue and red rays lie on separate rows.  The shaded interval
represents the values of nonzero quotient representatives.
The boundary circles are schematic: endpoint membership is determined
by the actual final segments.}
\label{fig:differentials-general}
\end{figure}
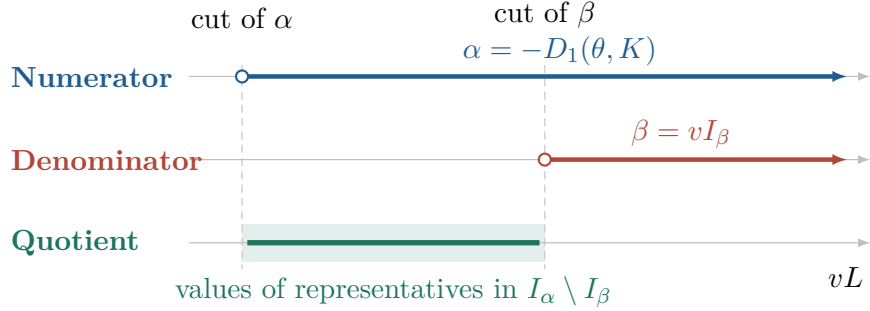
For Artin--Schreier defect extensions the denominator segment is $-pD$;
for Kummer defect generators satisfying $v\theta=0$ it is $vp-pD$.
In the latter case we assume characteristic zero, residue characteristic
$p$, and a primitive $p$th root of unity in $K$.
Here $D=v(\theta-K)$; Figure~\ref{fig:differentials-degree-p}
illustrates the two translations.
\begin{figure}[htbp]
\centering
\begin{tikzpicture}[value diagram]
\node[panel title] at (0,.7) {Artin--Schreier defect};
\fill[cutgreen!12] (3.2,-2.35) rectangle (6.6,-1.85);
\foreach \y in {0,-1,-2}
  \draw[value axis] (2.5,\y) -- (11.5,\y);
\node[row label,text=cutblue] at (0,0) {$I_\alpha$};
\node[row label,text=cutred] at (0,-1) {$I_\beta$};
\node[row label,text=cutgreen] at (0,-2) {$I_\alpha/I_\beta$};
\draw[value ray,cutblue] (3.27,0) -- (11.2,0);
\node[open cut,draw=cutblue] at (3.2,0) {};
\draw[value ray,cutred] (6.67,-1) -- (11.2,-1);
\node[open cut,draw=cutred] at (6.6,-1) {};
\draw[cutgreen,line width=2pt] (3.27,-2) -- (6.53,-2);
\node[formula label,text=cutblue] at (7.4,.32) {$\alpha=-D$};
\node[formula label,text=cutred] at (8.6,-.68) {$\beta=-pD$};
\draw[cutguide] (3.2,-.15) -- (3.2,-2.35);
\draw[cutguide] (6.6,-1.15) -- (6.6,-2.35);
\draw[black!15] (0,-2.9) -- (11.5,-2.9);
\node[panel title] at (0,-3.55) {Kummer defect};
\fill[cutgreen!12] (3.2,-6.6) rectangle (8.3,-6.1);
\foreach \y in {-4.25,-5.25,-6.25}
  \draw[value axis] (2.5,\y) -- (11.5,\y);
\node[row label,text=cutblue] at (0,-4.25) {$I_\alpha$};
\node[row label,text=cutpurple] at (0,-5.25) {$I_\beta$};
\node[row label,text=cutgreen] at (0,-6.25) {$I_\alpha/I_\beta$};
\draw[value ray,cutblue] (3.27,-4.25) -- (11.2,-4.25);
\node[open cut,draw=cutblue] at (3.2,-4.25) {};
\draw[value ray,cutpurple] (8.37,-5.25) -- (11.2,-5.25);
\node[open cut,draw=cutpurple] at (8.3,-5.25) {};
\draw[cutgreen,line width=2pt] (3.27,-6.25) -- (8.23,-6.25);
\node[formula label,text=cutblue] at (7.4,-3.93) {$\alpha=-D$};
\node[formula label,text=cutpurple] at (8.8,-4.93) {$\beta=vp-pD$};
\draw[cutguide] (3.2,-4.4) -- (3.2,-6.6);
\draw[cutguide] (8.3,-5.4) -- (8.3,-6.6);
\draw[cutguide] (6.6,-5.1) -- (6.6,-5.5);
\draw[shift arrow,cutpurple] (6.6,-7.05) -- (8.3,-7.05);
\node[below,text=cutpurple] at (7.45,-7.05) {$+vp$};
\node[anchor=east,font=\footnotesize] at (6.3,-7.05) {from the $-pD$ cut};
\end{tikzpicture}
\caption{Differential segments in degree $p$, with
$D=D_1(\theta,K)$.  Both numerator cuts equal $-D$.
The Artin--Schreier denominator is $-pD$; the Kummer denominator
has the additional translation by $vp$.
The green bands represent the quotients.  The cut positions and open
circles are schematic, including the possibility of coincident cuts.}
\label{fig:differentials-degree-p}
\end{figure}

\subsection{Ramification ideals}

For a pure generator set
\[
 \beta_H=\min_{\sigma\in H\setminus\{1\}}v(\sigma\theta-\theta).
\]
The following formula connects conjugate distances to relative changes
of arbitrary field elements.
\begin{theorem}
\label{thm:pure-ramification-ideal}
Let $(L/K,v)$ be a finite Galois extension, pure on $\theta$,
whose valuation is invariant under $G=\operatorname{Gal}(L/K)$
(as is automatic over a henselian base).  For every nontrivial $H\leq G$,
\begin{equation}\label{eq:IH}
 I_H=I_{\beta_H-D_1(\theta,K)}.
\end{equation}
\end{theorem}
The henselian formulation is stated in \cite{NovacoskiRam} as a
consequence of \cite{DuttaNovacoski}.  The following argument explains
why invariance of the chosen valuation suffices.
\begin{proof}
Fix $\sigma\ne1$ and write $\beta=v(\sigma\theta-\theta)$.
For $x=P(\theta)\ne0$, $\deg P<[L:K]$, choose $c$ with
$\nu_c(P)=vx$, and put $\delta=v(\theta-c)$.
Valuation invariance gives $v(\sigma\theta-c)=\delta$.
Expanding $P$ at $c$ and factoring the difference of two powers gives
\[
 v(\sigma x-x)\geq vx+\beta-\delta.
\]
Thus every generator $\sigma x/x-1$ belongs to
$I_{\beta-D_1(\theta,K)}$.  Conversely, taking $x=\theta-c$
realizes the value $\beta-v(\theta-c)$ for every $c\in K$.
This proves the formula for the ideal generated by one automorphism.
Summing over $\sigma\in H\setminus\{1\}$ selects the smallest
$\beta_\sigma$, proving \eqref{eq:IH}.
\end{proof}

In a Galois defect extension of degree $p$, all nontrivial conjugate
differences have the same value
\[
 \delta=v(\theta-\sigma(\theta)).
\]
There is only one nontrivial subgroup and hence one ramification ideal,
whose value segment is obtained by translating $-D$ by $\delta$.

\begin{figure}[htbp]
\centering
\begin{tikzpicture}[value diagram]
\node[panel title] at (0,.85) {One conjugate-distance level};
\node[anchor=east] at (11.5,.85) {$\delta=v(\theta-\sigma\theta)$, $\sigma\ne1$};
\foreach \y in {0,-2}
  \draw[value axis] (2.5,\y) -- (11.5,\y);
\node[row label,text=cutblue] at (0,0) {Base segment};
\node[row label,text=cutpurple] at (0,-2) {$vI_G$};
\draw[value ray,cutblue] (3.27,0) -- (11.2,0);
\node[open cut,draw=cutblue] at (3.2,0) {};
\draw[value ray,cutpurple] (6.27,-2) -- (11.2,-2);
\node[open cut,draw=cutpurple] at (6.2,-2) {};
\node[formula label,text=cutblue] at (7.4,.32) {$-D_1(\theta,K)$};
\node[formula label,text=cutpurple] at (8.7,-1.68) {$\delta-D_1(\theta,K)$};
\draw[shift arrow] (3.2,-.55) -- (6.2,-1.35);
\node[fill=white,inner sep=3pt] at (4.7,-.95) {translate by $\delta$};
\end{tikzpicture}
\caption{Prime-degree ramification translates the approximation segment
by the common conjugate-distance value $\delta$.
Open boundary markers are schematic.}
\label{fig:ramification-prime}
\end{figure}
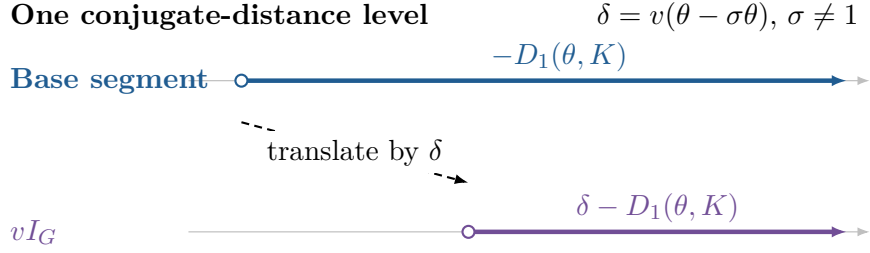

For a pure Galois extension of degree $n$, put
\[
 \beta_\sigma=v(\sigma\theta-\theta)
 \qquad(\sigma\ne1).
\]
Purity does not force these values to coincide.  Every distinct level that
is realized as a subgroup minimum translates the same final segment
$-D_1(\theta,K)$ and produces a ramification ideal.  

The differential quotient and the subgroup ramification ideals depend on
different operations on the same approximation segment.  The former
uses the derivative and truncation segment $\beta$, while the latter
uses translations by subgroup minima $\beta_H$.
Purity does not imply that the
conjugate-distance profile has a single level.

\begin{figure}[htbp]
\centering
\begin{tikzpicture}[value diagram]
\node[panel title] at (0,.75) {Several conjugate-distance levels};
\node[anchor=east] at (11.5,.75) {$\delta_1<\delta_2<\delta_3$};
\draw[cutguide] (3.2,.1) -- (3.2,-4.85);
\foreach \y in {0,-1.5,-3,-4.5}
  \draw[value axis] (2.5,\y) -- (11.5,\y);
\node[row label] at (0,0) {Base segment};
\draw[value ray,black!65] (3.27,0) -- (11.2,0);
\node[open cut,draw=black!65] at (3.2,0) {};
\node[formula label] at (7.4,.32) {$-D_1(\theta,K)$};
\foreach \y/\x/\idx/\col in {-1.5/5/1/cutblue,-3/6.7/2/cutgreen,-4.5/8.4/3/cutpurple}{
  \node[row label,text=\col] at (0,\y) {$vI_{H_\idx}$};
  \draw[value ray,\col] (\x+.07,\y) -- (11.2,\y);
  \node[open cut,draw=\col] at (\x,\y) {};
  \draw[shift arrow,\col] (3.2,\y+.48) -- (\x,\y+.48);
  \node[above,font=\footnotesize,text=\col] at ({(3.2+\x)/2},\y+.48) {$+\delta_\idx$};
  \node[formula label,text=\col] at (7.7,\y-.35)
    {$\delta_\idx-D_1(\theta,K)$};
}
\end{tikzpicture}
\caption{Each subgroup minimum $\delta_i$ translates the same base
segment into $vI_{H_i}$.
The positions are schematic, the value group need not be archimedean,
and the boundary markers do not prescribe endpoint membership.}
\label{fig:pure-branching}
\end{figure}
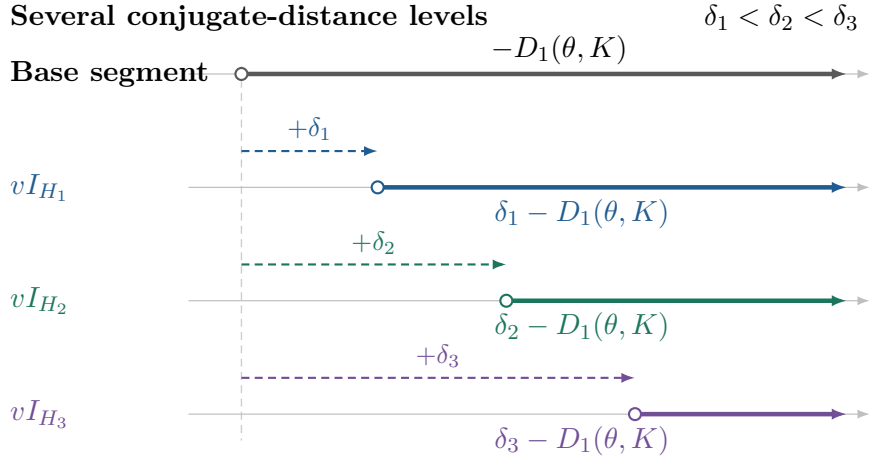

\section{The primitive-generator different and norm ideals}
\label{sec:naive-norms}

\subsection{The primitive-generator different}

The normalized linear presentation of the valuation ring gives a
derivative-generator formula for $\Dzero$.

\begin{theorem}
\label{thm:D0}
Assume that $(L/K,v)$ is finite, separable, immediate, unibranched, and pure
on $\theta$.  Then
\begin{equation}\label{eq:D0generators}
 \Dzero(\OL/\OK)
 =\left((\theta-c)^{1-n}g'(\theta):c\in K\right)\OL.
\end{equation}
Equivalently, its value final segment is the final segment generated by
\begin{equation}\label{eq:D0segment}
 v(g'(\theta))-(n-1)D_1(\theta,K).
\end{equation}
\end{theorem}

\begin{proof}
For $c\in K$, choose $t_c\in K$ with
$vt_c=-v(\theta-c)$ and put $\theta_c=t_c(\theta-c)$.  Purity gives
\(
\OL=\bigcup_c\OK[\theta_c].
\)
For any integral primitive generator $b$, choose $c$ such that
$b=F(\theta_c)$ with $F\in\OK[X]$.  For every nonidentity
$K$-embedding $\tau$ of $L$ into a valued algebraic closure,
$F(\theta_c)-F(\tau\theta_c)$ is divisible by
$\theta_c-\tau\theta_c$ with integral quotient.  Here unibranchedness
ensures that $\tau\theta_c$ is integral.  Since both generators are
primitive, multiplying over the same $n-1$ embeddings gives
$v h_b'(b)\geq v h_{\theta_c}'(\theta_c)$.
Thus $\Dzero$ is generated by the elements $h_{\theta_c}'(\theta_c)$.
The transformed minimal polynomial
$h_{\theta_c}(X)=t_c^n g(c+X/t_c)$ gives
\[
 h_{\theta_c}'(\theta_c)=t_c^{\,n-1}g'(\theta),
\]
whose value is $v(g'(\theta))-(n-1)v(\theta-c)$.
\end{proof}

The normalized-generator calculation appears in
\cite[Lemma 4.18]{KuhlmannRzepkaTopics}; the proof above establishes it
with our primitive-generator convention in arbitrary finite degree.  In particular,
\cite[Proposition 4.19]{KuhlmannRzepkaTopics} is the specialization of
Theorem~\ref{thm:D0} obtained by replacing the full derivative value by the
single conjugate-difference value available in degree $p$.

\begin{remark}[Limits of a scaled-generator presentation]
Purity in the purely ramified case does not alone imply that
$\OL$ is generated by all integral scalar multiples of a fixed
primitive element.  For example, take $K=k((t))$, with
$\operatorname{char}k\ne3$, and $\theta^3=t^2$.  The values of
$1,\theta,\theta^2$ are distinct modulo $vK$, so $\theta$ is pure.
Nevertheless, $vh<v\theta=2/3$ implies $vh\leq0$, and
\[
 \OK[\theta/h:vh<v\theta]=\OK[\theta]\ne\OL:
 \qquad \theta^2/t\in\OL\setminus\OK[\theta].
\]
Thus the derivative formula above is asserted for the pure immediate
case.  In a monogenic integral extension $\OL=\OK[b]$, the usual
formula $\Dzero=\DD=(h_b'(b))$ applies separately.
\end{remark}

\subsection{Conjugate-distance factorization}

For $\sigma\in G\setminus\{1\}$ put
\[
 \beta_\sigma=v(\sigma\theta-\theta).
\]
Separability gives the exact identity
\begin{equation}\label{eq:derivativesum}
 v(g'(\theta))
 =\sum_{\sigma\in G\setminus\{1\}}\beta_\sigma.
\end{equation}
The ramification ideal $I_G$ reflects only
\[
 \beta_G=\min_{\sigma\ne1}\beta_\sigma.
\]
Combining \eqref{eq:D0segment}, \eqref{eq:IH}, and
\eqref{eq:derivativesum} gives the product formula below.

For every $\sigma\in G\setminus\{1\}$ define the
\emph{conjugate-distance ideal}
\[
 J_\sigma(\theta/K):=
 I_{\beta_\sigma-D_1(\theta,K)}.
\]
In fact this ideal is intrinsic:
\[
 J_\sigma(\theta/K)=I_{\langle\sigma\rangle}=\mathcal J_\sigma.
\]
Indeed, telescoping and valuation invariance give
$v(\sigma^i\theta-\theta)\geq v(\sigma\theta-\theta)$, so the
subgroup minimum is $\beta_\sigma$.  Similarly,
$(\sigma^i x-x)/x$ is a sum of integral multiples of conjugates of
$(\sigma x-x)/x$, with the same values, whence
$\mathcal J_{\sigma^i}\subseteq\mathcal J_\sigma$.
The pure ramification formula proves the assertion.

\begin{theorem}
\label{thm:profile-factorization}
Under the hypotheses of Theorem~\ref{thm:D0}, assume additionally that $L/K$ is Galois.
\[
 \quad
 \Dzero(\OL/\OK)
 =\prod_{\sigma\in G\setminus\{1\}}J_\sigma(\theta/K).
 \quad
\]
Consequently,
\[
 \Dzero(\OL/\OK)\subseteq I_G^{\,n-1}.
\]
If
\[
 \varepsilon_G(\theta)=
 \sum_{\sigma\ne1}(\beta_\sigma-\beta_G),
\]
then equality holds precisely when translation by
$\varepsilon_G(\theta)$ fixes the final segment
$(n-1)(\beta_G-D_1(\theta,K))$.
\end{theorem}

\begin{proof}
The value segment of a product of fractional ideals of a valuation ring is
generated by the Minkowski sum of their value segments.  Thus the product on
the right has value segment
\[
 \operatorname{Final}\left(
 \sum_{\sigma\ne1}\beta_\sigma-(n-1)D_1(\theta,K)\right).
\]
Here the final segment generated by the $(n-1)$-fold Minkowski sum of
$-D_1(\theta,K)$ agrees with that generated by
$-(n-1)D_1(\theta,K)$: one inclusion is diagonal, and for the other one uses
the largest of finitely many elements of the initial segment
$D_1(\theta,K)$.  The asserted equality now follows from
\eqref{eq:derivativesum} and Theorem~\ref{thm:D0}.  Since
$\beta_\sigma\geq\beta_G$, one has $J_\sigma\subseteq I_G$ for every
$\sigma\ne1$.  The final assertion follows by comparing the two value
segments.
\end{proof}

In the product, each conjugate-distance ideal occurs with its
multiplicity.  The ideal $I_G$ is determined by the minimum distance.

In degree $p$, valuation invariance and the telescoping identity for
$\sigma^i\theta-\theta$ imply that all nontrivial conjugate differences
have the same value: applying the inequality in both directions, using
an inverse of $i$ modulo $p$, gives equality.  This yields the exponent
$p-1$ in the prime-degree factorization.  Purity alone does not impose
a common conjugate-distance value in higher degree.

\begin{corollary}
\label{cor:ramification-level-product}
Under the hypotheses of Theorem~\ref{thm:profile-factorization}, let
$I_1,\ldots,I_r$ be the distinct ideals $I_{\langle\sigma\rangle}$,
and let $m_j$ count the nonidentity automorphisms yielding $I_j$.
Then
\[
 \Dzero=\prod_{j=1}^r I_j^{m_j},\qquad
 \sum_{j=1}^r m_j=n-1.
\]
The factors and their multiplicities are independent of the chosen
pure generator.
\end{corollary}
\begin{proof}
Group the factors in Theorem~\ref{thm:profile-factorization}, using
$J_\sigma(\theta/K)=I_{\langle\sigma\rangle}$.
\end{proof}

\begin{corollary}
\label{cor:cut-stabilizer}
Under the hypotheses of Theorem~\ref{thm:profile-factorization},
put $S=v(I_G^{n-1})$ and
$H(S)=\{\gamma\in vL:S+\gamma=S\}$.  Then $H(S)$ is a convex
subgroup of $vL$, and
\[
 \Dzero=I_G^{n-1}\quad\Longleftrightarrow\quad
 \varepsilon_G(\theta)\in H(S).
\]
If $vL$ has rank one, equality holds if and only if all nonidentity
conjugate differences have the same value.
\end{corollary}
\begin{proof}
The stabilizer of a final segment is a subgroup.  If
$0\leq\eta\leq\gamma$ and $S+\gamma=S$, the inclusions
$S+\gamma\subseteq S+\eta\subseteq S$ show convexity.
The equality criterion is Theorem~\ref{thm:profile-factorization}.
Here $S$ is nonempty and proper: the approximation set has no maximum,
and valuation invariance bounds it above by every $\beta_\sigma$.
A rank-one ordered group has no nonzero proper convex subgroup, so
$H(S)=\{0\}$.  The nonnegative summands defining
$\varepsilon_G(\theta)$ sum to zero precisely when each is zero.
\end{proof}

\subsection{Norms of valuation ideals}

The value-theoretic part of the norm computation does not depend on prime
degree or purity.

\begin{proposition}
\label{prop:norm}
Let $(L/K,v)$ be a finite unibranched Galois extension of degree $n$.  For
every $x\in L$,
\[
 v\Nm_{L/K}(x)=n\,vx.
\]
Consequently, if an $\OL$-ideal $J$ is generated by $(x_i)_{i\in A}$, then
\[
 (\Nm_{L/K}J)\OK=(\Nm_{L/K}x_i:i\in A)\OK.
\]
If $J=I_S$, the value set of the norm ideal is the final segment of $vK$
generated by
\[
 \{ns:s\in S\}.
\]
In particular, for a pure Galois extension,
\begin{equation}\label{eq:normIH}
 v\bigl((\Nm_{L/K}I_H)\OK\bigr)
 =\operatorname{Final}_{vK}\bigl(n(\beta_H-D_1(\theta,K))\bigr).
\end{equation}
\end{proposition}

\begin{proof}
Unibranchedness implies $v(\sigma x)=vx$ for every $\sigma\in\Gal(L/K)$,
so the product formula for the norm gives $v\Nm(x)=n\,vx$.
If $x$ belongs to the ideal generated by the $x_i$, then $vx\geq vx_i$
for some $i$, because a valuation ring has linearly ordered ideals.
Hence $v\Nm(x)\geq v\Nm(x_i)$, and $\Nm(x)$ belongs to the ideal
generated by the $\Nm(x_i)$.  The reverse inclusion is immediate.
Passing to value segments proves the remaining assertions.
\end{proof}

The Artin--Schreier and Kummer formulas in
\cite[Propositions 4.21 and 4.22]{KuhlmannRzepkaTopics} are explicit-generator
versions of
\eqref{eq:normIH}; those polynomial expressions do not persist for a general
pure generator, but the final-segment formula does.

\section{Traces and the Dedekind different}
\label{sec:trace-profiles}

We use normalized generators and trace
duality to compare the Dedekind different with the primitive-generator
different.  Purity implies equality with the trace ideal; taking its
divisorial closure gives the Dedekind different.  Constant derivatives
then yield particularly explicit trace generators.

\subsection{Trace duality and divisorial closure}
\label{sec:dedekind-different}

For a fractional $\OL$-ideal $J$, put
\[
 J^v:=\OL:(\OL:J).
\]
This is the divisorial closure of $J$.  In rank one it has the
endpoint interpretation described in the preliminaries.

For an integral primitive element $b$ with minimal polynomial $h_b$,
the trace dual of $\OK[b]$ is
\[
 (\OK[b])^*=h_b'(b)^{-1}\OK[b].
\]
This standard monogenic calculation is also used in
\cite[Lemma 7.1]{ThatteAS}.

For every $x\in\CC$, the inclusion $\OK[b]\subseteq\OL$ places
$x$ in $(\OK[b])^*$, hence $h_b'(b)x\in\OK[b]\subseteq\OL$.
Thus $h_b'(b)\in\DD(\OL/\OK)$.  Since an inverse fractional ideal
is divisorial, it follows that
\begin{equation}\label{eq:D0v-in-D}
 \Dzero(\OL/\OK)^v\subseteq\DD(\OL/\OK).
\end{equation}
\begin{theorem}
\label{thm:pure-trace-different}
Let $(L/K,v)$ be a finite separable extension of degree $n>1$.
Assume that it is immediate, unibranched, and pure on $\theta$.
Put $J=\Dzero(\OL/\OK)$.  Then
\[
 \begin{aligned}
 T(\OL/\OK)&=J,\\
 \CC(\OL/\OK)&=\OL:_L J,\\
 \DD(\OL/\OK)&=J^v.
 \end{aligned}
\]
These assertions do not require a constant derivative or a rank-one
valuation.
\end{theorem}

\begin{proof}
We first prove a trace estimate.  For $z=P(\theta)\ne0$, with
$\deg P<n$, choose $c\in K$ such that the linear truncation recovers
both $P$ and $g'$:
\[
 \nu_c(P)=vz,\qquad \nu_c(g')=v(g'(\theta)).
\]
Such a common $c$ exists by purity and the monotonicity of linear
truncations with the approximation value.  Put
$\delta=v(\theta-c)$ and $Q=X-c$.  Unibranchedness implies that every
root of $g$ has distance $\delta$ from $c$.  Thus, in
$g=Q^n+\sum_{j<n}g_{j,c}Q^j$, each coefficient satisfies
$vg_{j,c}\geq(n-j)\delta$ and $\nu_c(g)=n\delta$.

Let $R$ be the remainder of $Pg'$ on division by $g$.
Replacing a term $bQ^m$, $m\geq n$, by
$-bQ^{m-n}\sum_{j<n}g_{j,c}Q^j$ never decreases its weighted value.
Euclidean division therefore gives
\[
 \nu_c(R)\geq\nu_c(Pg')=vz+v(g'(\theta)).
\]
Lagrange interpolation identifies $\Tr_{L/K}(z)$ with the
coefficient of $Q^{n-1}$ in $R$: at each root $\theta_i$,
$R(\theta_i)=P(\theta_i)g'(\theta_i)$, and the leading coefficient
of the interpolation polynomial is $\sum_i P(\theta_i)$.
Consequently, when the trace is nonzero,
\begin{equation}\label{eq:pure-trace-estimate}
 v\Tr_{L/K}(z)\geq vz+v(g'(\theta))-(n-1)\delta.
\end{equation}
For $z\in\OL$, the right side is at least the value of the generator
$t_c^{n-1}g'(\theta)$ of $J$.  Hence $T\subseteq J$.

For the reverse inclusion, take any integral primitive element $b$
with monic minimal polynomial $h\in\OK[X]$.  The identity
\[
 h'(X)=\Tr_{L/K}\left(\frac{h(X)}{X-b}\right)
\]
is coefficientwise: the quotient is a polynomial in $\OK[b][X]$,
and summing over embeddings differentiates the product of the linear
factors of $h$.  Every coefficient of $h'$ therefore belongs to
$\Tr_{L/K}(\OL)$.  Evaluating at $b\in\OL$ gives $h'(b)\in T$.
Thus $J\subseteq T$.

The complementary-module identity can also be checked directly.
We already have $\CC\subseteq\OL:J$ from
\eqref{eq:D0v-in-D}.  If $xJ\subseteq\OL$ and $y\in\OL$, apply
\eqref{eq:pure-trace-estimate} to $z=xy$.  Its right side is
nonnegative, since $vy\geq0$ and
$x t_c^{n-1}g'(\theta)\in\OL$.  Thus $\Tr(xy)\in\OK$ and
$x\in\CC$.  Inverting gives $\DD=J^v$.
\end{proof}

The passage from the trace ideal to the complementary module is also
\cite[Theorem 5.4]{CKR}: for finite unibranched extensions with
$vL=vK$, that theorem already gives $\CC=\OL:T$ and $\DD=T^v$.
Purity gives $T=\Dzero$; the identities for $\CC$ and $\DD$
agree with this more general result.

\subsection{Pure extensions with constant derivative}

\begin{corollary}
\label{thm:constant-derivative}
Under the hypotheses of Theorem~\ref{thm:pure-trace-different},
suppose that the monic minimal polynomial of $\theta$
satisfies $g'(X)=a\in K^\times$.  For the normalized generators
$u_c=t_c(\theta-c)$, one has
\[
 \Tr_{L/K}(\OL)=(a t_c^{n-1}:c\in K)\OK.
\]
\end{corollary}

\begin{proof}
For a monic separable polynomial $h$ of degree $n$ and a root $b$
of degree $n$, Lagrange interpolation gives
\[
 \Tr_{K(b)/K}\bigl(b^i/h'(b)\bigr)
 =\begin{cases}0,&0\leq i<n-1,\\1,&i=n-1.\end{cases}
\]
Apply this to $h(X)=g(X+c)$ and $b=\theta-c$.  Since $h'=a$,
\[
 \Tr((\theta-c)^i)=0\ (i<n-1),\qquad
 \Tr((\theta-c)^{n-1})=a.
\]
Thus the normalized generators $u_c=t_c(\theta-c)$ satisfy
\[
 \Tr(\OL)=\bigl(a t_c^{n-1}:c\in K\bigr)\OK.
\]
\end{proof}

\begin{corollary}
\label{cor:additive-pure}
In characteristic $p$, Corollary~\ref{thm:constant-derivative} applies
whenever a pure immediate unibranched extension is generated by a root
of an irreducible polynomial
\[
 X^{p^r}+a_{r-1}X^{p^{r-1}}+\cdots+a_1X^p+a_0X-b,
 \qquad a_0\ne0.
\]
Its degree can be any $p^r$ for which these hypotheses hold.
\end{corollary}
\begin{proof}
The derivative is the nonzero constant $a_0$.
\end{proof}
The corollary is conditional on purity and immediacy; an additive
presentation by itself implies neither property.

\subsection{Differential annihilators}

We work under the assumptions of
Theorem~\ref{thm:pure-trace-different} and set
\[
 C=\{\gamma\in(vL)_{\geq0}:\gamma+\alpha\subseteq\beta\},
\]
\[
 S_0=\operatorname{Final}\bigl(
 v(g'(\theta))-(n-1)D_1(\theta,K)\bigr).
\]
By Theorem~\ref{thm:pure-trace-different},
\[
 \DD(\OL/\OK)=\ann\Omega_{\OL/\OK}
 \quad\Longleftrightarrow\quad
 S_0^v=C,
\]
where $S_0^v=v((I_{S_0})^v)$ denotes the divisorial closure of the
value segment (an endpoint closure in rank one).  This reduces
the equality to a comparison of the derivative-distance cut with the
key-polynomial cut defining $\ann\Omega$.

\section{Swan ideals and logarithmic upper ramification}
\label{sec:swan}

Purity gives explicit formulas for the invariants considered above,
but the passage to Swan conductors requires additional information.
A Swan conductor is attached to a character or, more generally, a
representation.  Purity is a condition on approximation by elements of
$K$; it neither specifies such a representation nor implies that the
extension is Galois or cyclic of degree $p$.  Even for a pure Galois
extension, the norm of $I_H$ from the full field $L$ need not be the
Swan ideal of a character: the prime-degree norm theorem uses the
norm from the character's quotient field.  The two norms have
different degrees, as the example in
Section~\ref{subsec:balanced-upper-filtration} demonstrates.

There is nevertheless a general logarithmic consequence in the
immediate case: ordinary and relative logarithmic differentials agree
when the value groups agree.  We prove this below.  For conductors we
apply the existing prime-degree theory to character quotients and
recover the upper filtration when $G$ is elementary abelian.  This
reconstruction does not require purity.  In a general finite group,
order-$p$ characters do not suffice to separate subgroups; already
for $G\simeq\mathbb Z/p^2\mathbb Z$, all such characters vanish on
$pG$.  Higher-order characters require additional parameter and
conductor theory.  Thus the limitation here concerns the formulas
obtained from purity, rather than the existence of a general upper
ramification theory over henselian valuation rings
\cite{KatoThatte}.

\subsection{Prime-degree parameter ideals}

The parameter descriptions in this section are statements over
henselian valuation rings.  The equal-characteristic norm and
refined-conductor theorems of \cite{ThatteAS} assume either rank one
or defectlessness.  The mixed-characteristic results of
\cite{ThatteMixed} allow arbitrary rank, and Section~8 of that paper
extends the method to equal characteristic.

For henselian cyclic degree-$p$ extensions,
\cite[Theorem 0.3]{ThatteAS} and
\cite[Theorem 1.3]{ThatteMixed} identify the Swan ideal
$\mathcal H\subseteq\OK$ with
\begin{equation}\label{eq:thatte-swan-norm}
 \mathcal H=(\Nm_{L/K}\mathcal J_\sigma)\OK.
\end{equation}
In equal characteristic, $\mathcal H$ is generated by the reciprocals of
nonzero Artin--Schreier parameters defining the extension.  In the
Kummer case, with $z=\zeta_p-1$, it is generated by
$z^p/(h-1)$ for nonzero parameters $h$ defining the extension, with
$h\ne1$.  Thus Proposition~\ref{prop:norm} recovers the value-set part
of these theorems in the pure degree-$p$ cases:
\[
 v\mathcal H=
 \operatorname{Final}_{vK}\bigl(p(\delta-D_1(\theta,K))\bigr).
\]
Formula~\eqref{eq:normIH} extends this norm calculation to arbitrary
pure Galois degree and to each subgroup ideal.  It does not supply an
Artin--Schreier or Kummer parameter description of those higher-degree
norm ideals, nor a refined Swan conductor for them.

Proposition~3.10 and Corollary~4.5 of \cite{ThatteAS}, and Proposition~3.8 and
Corollary~4.4 of \cite{ThatteMixed}, relate defectlessness to
principality of the logarithmic ideal and of the Swan ideal (with the
stated hypotheses).  For a pure immediate unibranched extension of
arbitrary degree, $D_1(\theta,K)$ has no maximum, so every nontrivial
$I_H=I_{\beta_H-D_1(\theta,K)}$ is nonprincipal.  The same implication is not asserted for general higher-degree Galois
extensions without purity.

\subsection{Logarithmic differentials and the upper filtration}

\begin{proposition}\label{prop:ordinary-log-equal-values}
Let $(L/K,v)$ be an extension of valued fields with $vL=vK$.
The natural map
\[
 \Omega_{\OL/\OK}\longrightarrow\Omega^{\log}_{\OL/\OK}
\]
is an isomorphism.  In particular, for a finite separable, immediate,
unibranched extension pure on $\theta$,
\[
 \Omega^{\log}_{\OL/\OK}\simeq I_\alpha/I_\beta,
 \qquad
 \ann\Omega^{\log}_{\OL/\OK}=\ann\Omega_{\OL/\OK}.
\]
\end{proposition}

\begin{proof}
For $x\in L^\times$, choose $a\in K^\times$ with $va=vx$ and
write $x=au$, where $u\in\OL^\times$.  Define
\[
 \ell(x)=u^{-1}du\in\Omega_{\OL/\OK}.
\]
This does not depend on the choice of $a$: replacing $a$ by another
such element multiplies $u$ by a unit of $\OK$, whose differential
vanishes.  The Leibniz rule gives
$\ell(xy)=\ell(x)+\ell(y)$, and $\ell(a)=0$ for $a\in K^\times$.
If $x\in\OL\setminus\{0\}$, the chosen $a$ lies in $\OK$, so
$dx=a\,du=x\ell(x)$.  These are precisely the additional relations
for relative logarithmic differentials.  Consequently
$d\log x\mapsto\ell(x)$, together with $dx\mapsto dx$, defines an
inverse to the natural map: in the logarithmic module,
$d\log x=d\log u=u^{-1}du$.
The final assertion follows from
\eqref{eq:intro-differentials}.
\end{proof}

The proposition uses equality of value groups, not purity.  If the
value group increases, an arbitrary $x\in L^\times$ need not be a
$K$-multiple of an integral unit, so this proof does not apply.
The ordinary differential quotient alone therefore does not provide
the same identification for every pure extension.  Nor does the
identification above construct a refined Swan map, whose definition
also involves a character and its parameter ideal.

Theorems~0.5 and~1.5 of \cite{ThatteAS} and \cite{ThatteMixed}
construct a refined Swan conductor and an isomorphism
\[
 \Omega^{\log}_{\OL/\OK}/
 \mathcal J_\sigma\Omega^{\log}_{\OL/\OK}
 \simeq\mathcal J_\sigma/\mathcal J_\sigma^2,
\]
compatible with norm maps.  These are cyclic degree-$p$ statements
under the hypotheses specified above.  Proposition~\ref{prop:ordinary-log-equal-values}
identifies the differential modules when $vL=vK$, but does not extend
these particular character and norm maps to arbitrary degree.

\cite[Theorem 1.1]{KatoThatte} characterizes the Swan ideal by the
logarithmic upper ramification filtration.  If $\overline{\mathcal O}$
is the valuation ring in a separable closure and $\mathfrak a$ is a
nonzero proper ideal of it, the image of
$G_{\log}^{\mathfrak a}$ in a cyclic degree-$p$ quotient is the whole
quotient exactly when
\[
 \mathfrak a\cap\OK\supseteq\mathcal H;
\]
it is trivial otherwise.  Combined with
the prime-degree norm identity of \cite{ThatteAS,ThatteMixed}, this relates the norm of the prime-degree
logarithmic ramification ideal to an upper-filtration threshold.
The upper filtration and the subgroup ideals $I_H$ are distinct
constructions.  Section~\ref{subsec:balanced-swan} applies this criterion to the
depth-one defect example.

Finally, Section~1.5 and the example in Section~10.4 of
\cite{KatoThatte} show why all ideals, rather than just principal
ideals, are needed as indices: the cuts represented by
$a\overline{\mathcal O}$ and $a\mathfrak m_{\overline{\mathcal O}}$
can give different thresholds although principal test ideals do not
distinguish them.  Thus the full ideal index set retains endpoint information that is
absent from the positive principal indices.

\subsection{Reconstruction from character quotients}
Let $M/F$ be an elementary abelian $p$-extension of henselian valued
fields with residue characteristic $p$.  Put $A=\mathcal O_F$ and
$G=\Gal(M/F)$.
For a nonzero character $\chi:G\to\mathbb F_p$, let $\mathcal H_\chi$
be the Swan ideal of $M^{\ker\chi}/F$.
Let $\overline A$ be the valuation ring in a separable closure, and
let $U_{\mathfrak a}$ denote the image in $G$ of the logarithmic upper
group indexed by a nonzero proper ideal $\mathfrak a$ of $\overline A$.

\begin{proposition}
\label{prop:character-reconstruction}
With the above notation,
\begin{equation}\label{eq:character-upper-reconstruction}
 U_{\mathfrak a}=
 \bigcap_{\substack{\chi\ne0\\
          \mathfrak a\cap A\subsetneq\mathcal H_\chi}}\ker\chi,
\end{equation}
where an empty intersection denotes $G$.
\end{proposition}

\begin{proof}
By \cite[Theorem 1.1]{KatoThatte}, the image of $U_{\mathfrak a}$ in
$G/\ker\chi$ is trivial exactly when
$\mathfrak a\cap A\subsetneq\mathcal H_\chi$.
A subgroup of an elementary abelian group is the intersection of the
kernels of all characters vanishing on it.  Applying this elementary
duality proves the formula.
\end{proof}

\section{Pure extensions of degree \texorpdfstring{$p^2$}{p-squared}: two examples}
\label{sec:examples}

\subsection{A totally ramified Artin--Schreier compositum}

\begin{example}\label{ex:AS-totally-ramified}
Let $k$ be a perfect field of characteristic $p>0$, let $K=k((t))$
with $vt=1$, and put $n=p^2$.  Consider
\[
 \begin{gathered}
u^p-u=t^{-1},\\
 w^p-w=t^{-(p+1)},\\
 L=K(u,w).
 \end{gathered}
\]
Recall that replacing a root $y$ of $y^p-y=f$ by $y-a$, with
$a\in K$, replaces the parameter by $f-(a^p-a)$.  An element
$a^p-a$ is called an Artin--Schreier coboundary, and the parameter
classes lie in the $\mathbb F_p$-vector space
$K/\wp(K)$, where $\wp(K)=\{a^p-a:a\in K\}$.
If $v(a^p-a)<0$, then $va<0$ and
\[
 v(a^p-a)=p\,va,
\]
since $p\,va<va$.  Its pole order is therefore divisible by $p$.
On the other hand, for $(A,B)\in\mathbb F_p^2\setminus\{(0,0)\}$,
\[
 v\bigl(At^{-1}+Bt^{-(p+1)}\bigr)=
 \begin{cases}
 -1,&B=0,\\
 -(p+1),&B\ne0.
 \end{cases}
\]
Neither pole order is divisible by $p$.  Hence no nonzero combination
of the two parameters lies in $\wp(K)$, proving that their classes
are linearly independent.  Artin--Schreier theory now gives $[L:K]=n$ and
\[
 G=\{\rho_{i,j}:u\mapsto u+i,\ w\mapsto w+j\mid
                       i,j\in\mathbb F_p\}\simeq C_p\times C_p.
\]
This is a degree-$p^2$ Artin--Schreier compositum; a single
Artin--Schreier equation would have degree at most $p$.

An explicit pure generator is obtained by setting
\[
 z=w-u^{p+1},\qquad \theta=(tz)^{-1}.
\]
Indeed,
\[
 z^p-z=-ut^{-p}-ut^{-1}-t^{-2}.
\]
Since $vu=-1/p$, the first term has strictly smaller value than
the others.  Therefore
\[
 vz=-1-\frac1n,\qquad v\theta=\frac1n.
\]
It follows that $e(L/K)=n$, so $L/K$ is totally ramified,
$L=K(\theta)$, and $\theta$ is a uniformizer of $\OL$.
Its minimal polynomial $g$ is Eisenstein and
$\OL=\OK[\theta]$.  Moreover, for $f=\sum_{j=0}^{n-1}a_jX^j$,
the nonzero terms $a_j\theta^j$ have distinct values modulo $vK$.
Thus $v(f(\theta))=\min_j(va_j+j/n)$, which proves purity on
$\theta$.  The approximation set $D_1(\theta,K)$ has maximum $1/n$.

The conjugate distances can be computed directly.  We have
\[
 \rho_{i,j}z-z=j-2iu-it^{-1}-i^2,
\]
whose value is $-1$ if $i\ne0$, and zero if $i=0$, $j\ne0$.
This remains valid for $p=2$, when the term $2iu$ vanishes.
Since every conjugate of $z$ has value $-1-1/n$,
\[
 \beta_{i,j}=v(\rho_{i,j}\theta-\theta)
 =\begin{cases}
  2/n,&i\ne0,\\
  1+2/n,&i=0,\ j\ne0.
 \end{cases}
\]
There are $n-p$ conjugates of the first type and $p-1$ of the second.
Consequently, for $H=\{\rho_{0,j}:j\in\mathbb F_p\}$,
\[
 \begin{gathered}
I_G=\mathfrak m_L=(\theta)\OL,\\
 I_H=t\mathfrak m_L,\\
 I_J=\mathfrak m_L\quad(J\ne\{1\},\ J\not\subseteq H).
 \end{gathered}
\]
The minimum distance determines $I_G$, while the derivative retains
both distance levels:
\[
 v(g'(\theta))
 =(n-p)\frac2n+(p-1)\left(1+\frac2n\right)
 =p+1-\frac2n.
\]
Writing $d=(p+1)n-2$, monogenicity gives
\[
 \begin{gathered}
 \Dzero=\DD=\ann\Omega_{\OL/\OK}=\mathfrak m_L^d,\\
 \Omega_{\OL/\OK}\simeq\OL/\mathfrak m_L^d.
 \end{gathered}
\]
Since $d>n-1$, we obtain
\[
 \Dzero=\DD=\ann\Omega_{\OL/\OK}
             \subsetneqq I_G^{\,n-1}.
\]
The classical trace formula, with $e=n$ and different exponent $d$,
gives
\[
 \begin{aligned}
 \Tr_{L/K}(\OL)&=\mathfrak m_K^{\lfloor d/n\rfloor}=t^p\OK,\\
 T(\OL/\OK)&=t^p\OL\subsetneqq I_G^{\,n-1}.
 \end{aligned}
\]
The norm ideals are
\[
 (\Nm_{L/K}I_G)\OK=t\OK,\qquad
 (\Nm_{L/K}I_H)\OK=t^{n+1}\OK.
\]
Finally, the lower ramification breaks are $1$ and $n+1$, since
for a uniformizer the lower index is $n\beta_{i,j}-1$.
Herbrand's change of numbering gives upper breaks $1$ and $p+1$.
Equivalently, the character $\chi_{a,b}(\rho_{i,j})=ai+bj$
is represented by $at^{-1}+bt^{-(p+1)}$, and its Swan conductor is
$1$ when $b=0$, $a\ne0$, and $p+1$ when $b\ne0$.
These are the classical Artin--Schreier and upper-numbering
computations of \cite{SerreLocal}.
\end{example}

This example disproves, for pure extensions of general degree, each naive
replacement of $p-1$ by $n-1$ in
\[
 \begin{gathered}
T(\OL/\OK)=I_G^{p-1},\\
 \Dzero=I_G^{p-1},\\
 \DD=I_G^{p-1},\\
 \ann\Omega_{\OL/\OK}=I_G^{p-1}.
 \end{gathered}
\]
In this example $\DD=\ann\Omega$, although neither ideal equals the
proposed power of the largest subgroup ramification ideal.

\subsection{An immediate elementary abelian extension}
\label{subsec:balanced-cut-example}

To examine purity in the presence of defect beyond prime degree, we use
the depth-one Hahn-field construction of
\cite[Section 5.2]{NovacoskiDistances}.  Let
$k=\overline{\mathbb F_p}$ and choose $r_1=p<r_2<\cdots$ tending to
infinity such that $\{\pi\}\cup\{1/r_i:i\geq1\}$ is linearly
independent over $\mathbb Q$, where $\pi$ is a fixed positive irrational real number.  In $k((t^{\mathbb R}))$, put
\[
 \begin{gathered}
a_\ell=t^{-\pi/p^\ell},\\
 b_\ell=t^{-1/r_\ell}-t^{-1/r_{\ell+1}},\\
 K=k(t,t^\pi)(a_\ell,b_\ell:\ell\geq1).
 \end{gathered}
\]
Write
$c_\ell=\sum_{j=1}^{\ell}a_j$ and
$d_\ell=\sum_{j=1}^{\ell}b_j$.

Let $u$ and $w$ be the negative-support Hahn representatives of the
Artin--Schreier roots below.  With $vt=1$, put
\[
 \begin{gathered}
u^p-u=t^{-\pi},\\
 w^p-w=t^{-p-1},\\
 \theta=u+tw,\\
 L=K(\theta)=K(u,w).
 \end{gathered}
\]
Thus $K$ has \emph{no finite limits}: the support of each of its
elements meets every interval $(-\infty,M)$ in a finite set.  The two degree-$p$
extensions are linearly disjoint, and the approximation calculations of \cite[Section 5.2]{NovacoskiDistances} give
\[
 \begin{gathered}
d_1(u,K)=0^-,\\
 d_1(w,K)=-1^-,\\
 d_1(\theta,K)=0^-,\\
 \operatorname{dep}(\theta)=1.
 \end{gathered}
\]
In particular, $L/K$ is pure on $\theta$ and has degree $p^2$.
The approximation cut is balanced because
$d_1(u,K)=vt+d_1(w,K)$.  For the finite approximants $c_\ell$ and
$d_\ell$ of that construction,
\[
 v\bigl(\theta-(c_\ell+d_\ell)\bigr)
 =\min\left\{-\frac{\pi}{p^{\ell+1}},-\frac1{r_{\ell+1}}\right\}
 \longrightarrow0^-.
\]
No approximation from $K$ has nonnegative value, since the negative
support of $\theta$ is infinite.  The group $vK$ is dense in
$\mathbb R$, since it contains $\mathbb Z+\pi\mathbb Z$.

The field $K$ is not assumed henselian.  We verify the valuation
hypotheses before using the general formulas.
\begin{lemma}
The chosen Hahn valuation makes $L/K$ immediate and unibranched.
\end{lemma}
\begin{proof}
For $z=P(\theta)\ne0$ of degree less than $p^2$, choose $c\in K$
with $\nu_c(P)=vz$.  Choose a better approximation $d\in K$, so that
$\delta_d=v(\theta-d)>\delta_c=v(\theta-c)$.
In the expansion $P(X)=\sum_j P_{j,c}(X-c)^j$, the difference of
powers satisfies
\[
 v\bigl((\theta-c)^j-(d-c)^j\bigr)
 \geq (j-1)\delta_c+\delta_d>j\delta_c\qquad(j\geq1).
\]
Hence $v(z-P(d))>vz$.  Every nonzero element of $L$ consequently
has an approximation in $K$ of the same value and the same leading
residue.  This proves $vL=vK$ and $Lv=Kv$.

For $\sigma\in G$, we have $v(\sigma\theta-\theta)\geq0$,
whereas $v(\theta-c)<0$ for every $c\in K$.  Therefore
$v(\sigma\theta-c)=v(\theta-c)$, and the same expansion gives
$v(\sigma z)\geq\nu_c(P)=vz$.  Iterating $\sigma$, which has
finite order, forces equality.  The valuation is thus $G$-invariant.
Since the Galois group acts transitively on the prolongations of a
valuation to a finite normal extension, there is only one such
prolongation.  This proves unibranchedness.
\end{proof}
Write $\Gamma=vK=vL$.  The invariant-valuation formulation of
Theorem~\ref{thm:pure-ramification-ideal} therefore applies directly.

\begin{theorem}
\label{thm:main-defect-example}
For the extension just constructed, $[L:K]=p^2$, $vL=vK$,
$Lv=Kv$, and $\theta$ has depth one.  The chosen valuation is
unibranched and the defect is $p^2$.  Write
$H_1=\{\rho_{0,j}:j\in\mathbb F_p\}$, where
$\rho_{i,j}(\theta)=\theta+i+tj$.  Then
\[
 \begin{aligned}
 \Omega_{\OL/\OK}&\simeq\mL/t^{p-1}\mL,\\
 \ann\Omega_{\OL/\OK}&=t^{p-1}\OL,\\
 I_{H_1}&=t\mL,\\
 I_H&=\mL\quad(\{1\}\ne H\ne H_1),\\
 \Dzero=T&=t^{p-1}\mL,\\
 \Tr_{L/K}(\OL)&=t^{p-1}\mathfrak m_K,\\
 \CC&=t^{1-p}\OL,\\
 \DD&=t^{p-1}\OL.
 \end{aligned}
\]
Thus the primitive-generator different and trace ideal have open
value cuts, whereas the Dedekind different and differential
annihilator have closed cuts.

Let $F=K^h$, $E=LF$, and $A=\mathcal O_F$.  The degree and defect
remain $p^2$.  The character
$\chi_{a,b}(\rho_{i,j})=ai+bj$ has Swan ideal
\[
 \mathcal H_{a,b}=\begin{cases}
 \mathfrak m_A,&b=0, a\ne0,\\
 t^p\mathfrak m_A,&b\ne0.
 \end{cases}
\]
For a nonzero proper ideal $\mathfrak a$ in the valuation ring of a
separable closure of $F$, the logarithmic upper image is $G$ if
$\mathfrak a\cap A\supseteq\mathfrak m_A$, is $H_1$ if
$t^p\mathfrak m_A\subseteq\mathfrak a\cap A\subsetneq\mathfrak m_A$,
and is trivial otherwise.  In contrast, the full-extension norm
ideals of $I_G$ and $I_{H_1}$ are respectively $\mathfrak m_A$ and
$t^{p^2}\mathfrak m_A$.
\end{theorem}

The differential, trace, different, and subgroup-ideal computations
are proved in the following subsections.  Proposition~\ref{prop:balanced-swan}
also verifies degree preservation under henselization and computes the
character ideals; Section~\ref{subsec:balanced-upper-filtration}
proves the filtration and norm assertions.  These computations together
prove Theorem~\ref{thm:main-defect-example}.

\subsubsection{The minimal polynomial and differential module}
Let $f=t^{-\pi}$, $h=t^{-p-1}$, and $q=t^p-t$.  Since
\[
 \theta^p-\theta-f=q w+t^p h,
\]
the monic minimal polynomial of $\theta$ is
\[
 \begin{aligned}
 Y(X)&=X^p-X-f-t^ph,\\
 g(X)&=Y(X)^p-q^{p-1}Y(X)-q^ph.
 \end{aligned}
\]
It has degree $p^2$ and vanishes at $\theta$, so it is indeed the minimal
polynomial.  Its derivative is constant:
\[
 \begin{gathered}
g'(X)=q^{p-1},\\
 vq=1,\\
 v(g'(\theta))=p-1.
 \end{gathered}
\]
Equivalently, the conjugates are $\theta+i+tj$, with
$i,j\in\mathbb F_p$, and
\[
 v(i+tj)=
 \begin{cases}
 0,&i\ne0,\\
 1,&i=0,\ j\ne0.
 \end{cases}
\]
There are $p(p-1)$ differences of value zero and $p-1$ of value one,
which again gives the derivative sum $p-1$.

For $a\in K$ put $d=v(\theta-a)<0$ and $Q=X-a$.
Every root $\theta+i+tj$ has distance $d$ from $a$.  Factoring $g$ shows
$v(g(a))=p^2d$.  The coefficient of $Q^k$ has value at least
$(p^2-k)d$, since it is a sum of products of $p^2-k$ such root
differences.  Consequently every term of the $Q$-expansion has
truncated value at least $p^2d$, with equality for the constant and
leading terms.  Hence
\[
 \begin{gathered}
\nu_Q(g)=p^2d,\\
 \alpha_Q=-d,\\
 \beta_Q=(p-1)-p^2d.
 \end{gathered}
\]
Taking the final segments generated by a cofinal linear family yields
\begin{equation}\label{eq:balanced-differential-cuts}
 \begin{aligned}
 \alpha&=\Gamma_{>0},&I_\alpha&=\mL,\\
 \beta&=\Gamma_{>p-1},&I_\beta&=t^{p-1}\mL.
 \end{aligned}
\end{equation}
Therefore the K\"ahler module and its annihilator are
\begin{equation}\label{eq:balanced-differentials}
 \begin{aligned}
 \Omega_{\OL/\OK}&\simeq\mL/t^{p-1}\mL\ne0,\\
 \ann_{\OL}\Omega_{\OL/\OK}&=t^{p-1}\OL.
 \end{aligned}
\end{equation}
Indeed, $\gamma+\Gamma_{>0}\subseteq\Gamma_{>p-1}$ holds exactly when
$\gamma\geq p-1$.  The strict cuts defining the numerator and denominator
therefore give a closed cut for the annihilator.  Nonzero quotient
representatives have values $0<vx\leq p-1$.

\begin{figure}[htbp]
\centering
\begin{tikzpicture}[value diagram]
\node[panel title] at (0,.8) {The differential module and its annihilator};
\draw[cutguide] (3.2,.1) -- (3.2,-3.55);
\draw[cutguide] (7.2,.1) -- (7.2,-3.55);
\foreach \y in {0,-1.1,-2.2,-3.3}
  \draw[value axis] (2.5,\y) -- (11.5,\y);
\node[row label,text=cutblue] at (0,0) {$I_\alpha$};
\node[row label,text=cutred] at (0,-1.1) {$I_\beta$};
\node[row label,text=cutgreen] at (0,-2.2) {$I_\alpha/I_\beta$};
\node[row label,text=cutpurple] at (0,-3.3) {$\ann\Omega$};
\draw[value ray,cutblue] (3.27,0) -- (11.2,0);
\node[open cut,draw=cutblue] at (3.2,0) {};
\draw[value ray,cutred] (7.27,-1.1) -- (11.2,-1.1);
\node[open cut,draw=cutred] at (7.2,-1.1) {};
\fill[cutgreen!12] (3.2,-2.42) rectangle (7.2,-1.98);
\draw[cutgreen,line width=2pt] (3.27,-2.2) -- (7.2,-2.2);
\node[open cut,draw=cutgreen] at (3.2,-2.2) {};
\node[closed cut,draw=cutgreen,fill=cutgreen] at (7.2,-2.2) {};
\draw[value ray,cutpurple] (7.2,-3.3) -- (11.2,-3.3);
\node[closed cut,draw=cutpurple,fill=cutpurple] at (7.2,-3.3) {};
\node[formula label,text=cutblue] at (7.5,.32) {$\alpha=\Gamma_{>0}$};
\node[formula label,text=cutred] at (9,-.78) {$\beta=\Gamma_{>p-1}$};
\node[formula label,text=cutgreen] at (8.7,-1.88) {$0<vx\leq p-1$};
\node[formula label,text=cutpurple] at (9,-2.98) {$v(\ann\Omega)=\Gamma_{\geq p-1}$};
\node[below] at (3.2,-3.55) {$0$};
\node[below] at (7.2,-3.55) {$p-1$};
\node[open cut] at (3.2,-4.35) {};
\node[anchor=west,font=\footnotesize] at (3.4,-4.35) {endpoint omitted};
\node[closed cut] at (7.2,-4.35) {};
\node[anchor=west,font=\footnotesize] at (7.4,-4.35) {endpoint included};
\end{tikzpicture}
\caption{The immediate elementary abelian example has
$\Omega\simeq\mL/t^{p-1}\mL$.
The denominator omits its lower endpoint, so a quotient representative
can have value $p-1$.  The annihilator $t^{p-1}\OL$ includes that
endpoint.  All four rows use the same value coordinates.}
\label{fig:balanced-differentials}
\end{figure}
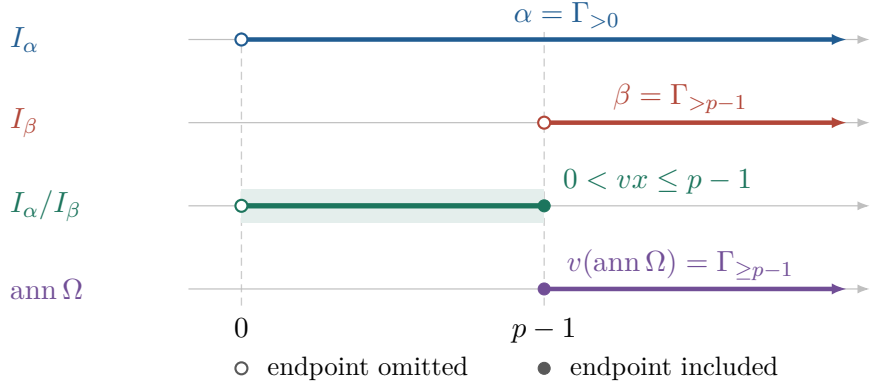

\subsubsection{Trace ideals and the two differents}

The derivative $g'=q^{p-1}$ is constant of value $p-1$.
The approximation values $\delta_c=v(\theta-c)$ are cofinal below
zero.  Theorem~\ref{thm:D0} therefore gives
\[
 \Dzero=t^{p-1}\mL.
\]
For completeness, the trace calculation is equally explicit.
Writing $n=p^2$ and $u_c=t_c(\theta-c)$, the only nonzero power trace
among $0\leq i<n$ is
\[
 \Tr(u_c^{n-1})=t_c^{n-1}q^{p-1},\qquad
 v\Tr(u_c^{n-1})=(p-1)-(n-1)\delta_c.
\]
These values are coinitial above $p-1$, and hence
\[
 \Tr_{L/K}(\OL)=t^{p-1}\mathfrak m_K,\qquad
 T=t^{p-1}\mL=\Dzero.
\]
Theorem~\ref{thm:pure-trace-different} yields
\[
 \begin{aligned}
 \CC&=\OL:_L(t^{p-1}\mL)=t^{1-p}\OL,\\
 \DD&=\Dzero^v=t^{p-1}\OL.
 \end{aligned}
\]
Together with \eqref{eq:balanced-differentials}, this proves
\[
 T=\Dzero\subsetneqq\DD=\ann\Omega_{\OL/\OK}.
\]
The inclusion is strict because the endpoint is unattained in
$v\Dzero=\Gamma_{>p-1}$.

The all-elements convention has a different answer here.  The
subextension $K(u)/K$ has approximation cut $0^-$; its integral
normalized generators have derivative values coinitial above zero.
Their derivatives therefore generate $\mL$ after extension to $\OL$.
Conversely, every integral $b\in L\setminus K$ and every distinct
$K$-conjugate $b'$ have the same residue: the Galois action preserves
the valuation and acts trivially on $Lv=Kv$.  Thus $v(b-b')>0$,
and the product formula for the minimal-polynomial derivative gives
$v(h_b'(b))>0$.  Consequently
\[
 \mathfrak D_{\mathrm{all}}=\mL
 \supsetneqq\Dzero=t^{p-1}\mL.
\]
Hence the two derivative-ideal conventions differ in composite degree.

\subsubsection{Subgroup ramification ideals}
Write $G=\Gal(L/K)\simeq\mathbb F_p^2$ and let
$\rho_{i,j}$ send $u$ to $u+i$ and $w$ to $w+j$.  Then
$\rho_{i,j}(\theta)-\theta=i+tj$.
The subgroup
\[
 H_1=\{\rho_{0,j}:j\in\mathbb F_p\}
 =\Gal(L/K(u))
\]
is the unique order-$p$ subgroup whose nonidentity elements all have
conjugate-distance value one.  Every other nontrivial subgroup contains
an element with $i\ne0$.  Thus
\[
 \beta_H=
 \begin{cases}
 1,&H=H_1,\\
 0,&H\ne H_1,\ H\ne\{1\}.
 \end{cases}
\]
Since $D_1(\theta,K)$ has cut $0^-$, the pure ramification formula gives
\begin{equation}\label{eq:balanced-ramification}
 \begin{aligned}
 I_{H_1}&=I_{\Gamma_{>1}}=t\mL,\\
 I_H&=I_{\Gamma_{>0}}=\mL\quad(H\ne H_1,\ H\ne\{1\}).
 \end{aligned}
\end{equation}
Let $\operatorname{Ram}(L/K,v)=\{I_H:\{1\}\ne H\leq G\}$.
Then
\[
 \begin{gathered}
\operatorname{Ram}(L/K,v)=\{t\mL,\mL\},\\
 t\mL\subsetneq\mL,
\#\operatorname{Ram}(L/K,v)=2>1=\operatorname{dep}(\theta).
 \end{gathered}
\]
Both ideals are nonprincipal because their value segments have no least
element.  The ramification calculation separates the two conjugate
levels, whereas the differential denominator combines their
multiplicities through $v(g'(\theta))=p-1$.  In this example one also has
$I_\beta=I_{H_1}^{p-1}=t^{p-1}\mL$, since
$\mL^{p-1}=\mL$ in the dense value group.

\begin{figure}[htbp]
\centering
\begin{tikzpicture}[value diagram]
\node[panel title] at (0,.8) {Two subgroup ramification ideals};
\draw[cutguide] (3.2,.1) -- (3.2,-3.2);
\draw[cutguide] (6.8,-2.55) -- (6.8,-3.2);
\foreach \y in {0,-1.4,-2.8}
  \draw[value axis] (2.5,\y) -- (11.5,\y);
\node[row label] at (0,0) {Base segment};
\node[row label,text=cutblue] at (0,-1.4) {$I_H=\mL$};
\node[row label,text=cutgreen] at (0,-2.8) {$I_{H_1}=t\mL$};
\draw[value ray,black!65] (3.27,0) -- (11.2,0);
\node[open cut,draw=black!65] at (3.2,0) {};
\draw[value ray,cutblue] (3.27,-1.4) -- (11.2,-1.4);
\node[open cut,draw=cutblue] at (3.2,-1.4) {};
\draw[value ray,cutgreen] (6.87,-2.8) -- (11.2,-2.8);
\node[open cut,draw=cutgreen] at (6.8,-2.8) {};
\node[formula label] at (7.4,.32) {$-D_1(\theta,K)=\Gamma_{>0}$};
\node[formula label,text=cutblue] at (7.4,-1.08)
  {$\beta_H=0\quad(H\ne H_1,\ H\ne\{1\})$};
\node[formula label,text=cutgreen] at (9,-2.48) {$vI_{H_1}=\Gamma_{>1}$};
\draw[shift arrow,cutgreen] (3.2,-2.25) -- (6.8,-2.25);
\node[above,text=cutgreen,font=\footnotesize] at (5,-2.25) {$+1$};
\node[below] at (3.2,-3.2) {$0$};
\node[below] at (6.8,-3.2) {$1$};
\node[anchor=west,font=\footnotesize] at (2.5,-3.9)
  {$H_1=\Gal(L/K(u))$,\qquad $t\mL\subsetneq\mL$};
\end{tikzpicture}
\caption{The subgroup minima zero and one give the open rays
$v\mL=\Gamma_{>0}$ and $v(t\mL)=\Gamma_{>1}$.
The subgroup $H_1$ gives the translated segment $\Gamma_{>1}$.}
\label{fig:balanced-ramification}
\end{figure}
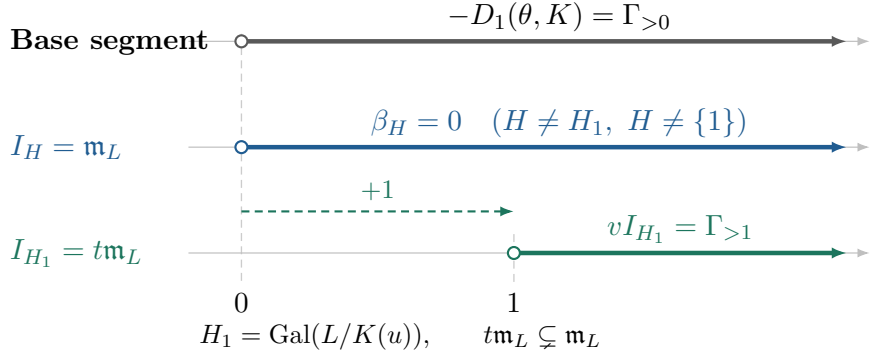

\subsubsection{Henselization and character Swan ideals}
\label{subsec:balanced-swan}

Let $K$, $u$, $w$, and $L=K(u,w)$ be as in
Section~\ref{subsec:balanced-cut-example}, and put
\[
 \begin{gathered}
F=K^h,\\
 E=LF,\\
 A=\mathcal O_F.
 \end{gathered}
\]
The henselization is taken with respect to the chosen Hahn valuation.
A rank-one valued field is dense in its henselization
\cite{KuhlmannHenselApprox}.  If $z$ has bounded approximation values
over $K$, it therefore lies outside $F$ and has the same approximation
cut over $F$: for $a\in F$, choose $a_0\in K$ with
$v(a-a_0)>v(z-a)$, giving $v(z-a_0)=v(z-a)$.

\begin{proposition}
\label{prop:balanced-swan}
The extension $E/F$ is immediate Galois of degree $p^2$.  Its
$p+1$ degree-$p$ subextensions are
\[
 E_\infty=F(u),\qquad E_\lambda=F(w+\lambda u)
 \quad(\lambda\in\mathbb F_p).
\]
They have Artin--Schreier parameters
\[
 f_\infty=t^{-\pi},\qquad
 f_\lambda=t^{-p-1}+\lambda t^{-\pi},
\]
approximation cuts $0^-$ and $-1^-$, respectively, and Swan ideals
\begin{equation}\label{eq:balanced-swan-ideals}
 \mathcal H_\infty=\mathfrak m_A,\qquad
 \mathcal H_\lambda=t^p\mathfrak m_A
 \quad(\lambda\in\mathbb F_p).
\end{equation}
All these subextensions have defect $p$.  Both Swan value segments
have unattained lower endpoints.
\end{proposition}

\begin{proof}
For $\lambda\ne0$, choose $c\in K$ with $v(u-c)>-1$.
An approximation to $w+\lambda u$ of value at least $-1$ would,
after subtracting $\lambda c$, approximate $w$ to value at least
$-1$, contradicting $d_1(w,K)=-1^-$.
Conversely, put $z_{\lambda,\ell}=t^{-1}d_\ell+\lambda c_\ell$.
For sufficiently large $\ell$,
\[
 v(w+\lambda u-z_{\lambda,\ell})
 =-1-\frac1{r_{\ell+1}}\longrightarrow-1^-.
\]
For $\lambda=0$ this is the calculation for $w$.
Density preserves these cuts over $F$, as well as the cut $0^-$ for
$u$.  Every nonzero $\mathbb F_p$-linear combination of $u$ and $w$
has bounded approximation values, so no such combination belongs to
$F$.  Their Artin--Schreier classes remain independent over $F$.
It follows that $[E:F]=p^2$ and that the stated list of subextensions
is complete.  The decomposition group of the chosen valuation on
$L/K$ is consequently all of $G$; hence $E=LK^h=L^h$.
Since $L/K$ and henselizations are immediate, $E/F$ is immediate.
Every degree-$p$ intermediate extension therefore has defect $p$.

For $(a,b)\in\mathbb F_p^2\setminus\{(0,0)\}$, write
\[
 \begin{gathered}
z_{a,b}=au+bw,\\
 f_{a,b}=a t^{-\pi}+b t^{-p-1},\\
 \rho_{a,b}=\sup v(z_{a,b}-F).
 \end{gathered}
\]
Then $\rho_{a,b}=0$ if $b=0$ and $\rho_{a,b}=-1$ otherwise.
The character $\chi_{a,b}(\rho_{i,j})=ai+bj$ is represented by the
Artin--Schreier class of $f_{a,b}$.
The parameters representing this fixed character are
$f_{a,b}-(c^p-c)$, $c\in F$, and satisfy
\[
 v\bigl(f_{a,b}-(c^p-c)\bigr)=p\,v(z_{a,b}-c),
\]
since the latter approximation values are negative.
The parameter description of the Swan ideal gives
\begin{equation}\label{eq:balanced-parameter-ideal}
 \mathcal H_{a,b}=
 \left(\frac1{f_{a,b}-(c^p-c)}:c\in F\right)A
 =\{0\}\cup\{x\in F^\times:vx>-p\rho_{a,b}\}.
\end{equation}
Nonzero scalar multiples of a character give the same ideal.
This proves \eqref{eq:balanced-swan-ideals}.
For explicit cofinal parameter families, take $c=c_\ell$ on the
independent line and $c=z_{\lambda,\ell}$ on the other lines.
The reciprocal parameter values are
\[
 \frac{\pi}{p^\ell}\longrightarrow0^+,
 \qquad p+\frac{p}{r_{\ell+1}}\longrightarrow p^+,
\]
respectively.  Neither limiting endpoint is attained.
\end{proof}

The prime-degree logarithmic ideals and ordinary differential modules
are therefore
\[
 \begin{aligned}
 \mathcal J_\infty&=\mathfrak m_{E_\infty},\\
 \Omega_{\mathcal O_{E_\infty}/A}&=0,\\
 \mathcal J_\lambda&=t\mathfrak m_{E_\lambda},\\
 \Omega_{\mathcal O_{E_\lambda}/A}
 &\simeq t\mathfrak m_{E_\lambda}/t^p\mathfrak m_{E_\lambda}.
 \end{aligned}
\]
Their norms to $F$ are the ideals of
\eqref{eq:balanced-swan-ideals}, in agreement with
the prime-degree norm identity of \cite{ThatteAS,ThatteMixed}.  The differential annihilators are
$\mathcal O_{E_\infty}$ and $t^{p-1}\mathcal O_{E_\lambda}$.
The independent line has Swan cut of infimum zero, whereas the $p$
dependent lines have Swan cut of infimum $p$.  These numerical
infima do not encode the unattained endpoints of the conductor ideals.

The refined Swan maps of \cite[Theorem 0.5]{ThatteAS} send
\[
 \left[\frac1{f_{a,b}-(c^p-c)}\right]
 \longmapsto\left[d\log\bigl(f_{a,b}-(c^p-c)\bigr)\right]
\]
from $\mathcal H_{a,b}/\mathcal H_{a,b}^2$ to
$\Omega_A^{\log}/(\mathcal I_{a,b}\cap A)\Omega_A^{\log}$,
where $\mathcal I_{a,b}$ is the ordinary difference ideal of the
corresponding degree-$p$ subextension.
For $b=0$ the source is zero, since $\mathfrak m_A^2=\mathfrak m_A$.
For $b\ne0$ it is $t^p\mathfrak m_A/t^{2p}\mathfrak m_A$.
These formulas determine the sources of the refined maps.  Their
images are not computed here.

\subsubsection{The logarithmic upper filtration and norm ideals}
\label{subsec:balanced-upper-filtration}

For the extension $E/F$ of Proposition~\ref{prop:balanced-swan}, put
$H_1=\ker\chi_{1,0}=\{\rho_{0,j}:j\in\mathbb F_p\}$.
Equations~\eqref{eq:balanced-swan-ideals} and
\eqref{eq:character-upper-reconstruction} give
\begin{equation}\label{eq:balanced-upper-filtration}
 U_{\mathfrak a}=
 \begin{cases}
 G,&\mathfrak a\cap A\supseteq\mathfrak m_A,\\
 H_1,&t^p\mathfrak m_A\subseteq\mathfrak a\cap A
                   \subsetneq\mathfrak m_A,\\
 \{1\},&\mathfrak a\cap A\subsetneq t^p\mathfrak m_A.
 \end{cases}
\end{equation}
For principal indices $\mathfrak a_r=\{x\in\overline A:vx\geq r\}$,
$r\in v(\overline F^\times)$, $r>0$, the image is $H_1$ for $0<r\leq p$ and is trivial for $r>p$.
The full group occurs at the nonprincipal index
$\mathfrak m_{\overline A}$.  At $r=p$ the closed index contains the
Swan ideal with the open cut $v>p$.

The subgroup norm ideals of the full extension have different
thresholds.  Their value segments are preserved by henselization:
$L$ is dense in $E=L^h$, and a sufficiently close $y\in L$ to
$x\in E^\times$, with $\sigma x\ne x$, satisfies
$v((\sigma y-y)/y)=v((\sigma x-x)/x)$.
Hence \eqref{eq:balanced-ramification} and
Proposition~\ref{prop:norm} give
\[
 (\Nm_{E/F}I_G)A=\mathfrak m_A,\qquad
 (\Nm_{E/F}I_{H_1})A=t^{p^2}\mathfrak m_A.
\]
The positive character Swan ideals instead equal $t^p\mathfrak m_A$,
because their norm maps originate in the degree-$p$ quotient fields.
In particular, the norm of $I_G$ alone does not determine the positive
upper threshold.

\section*{Acknowledgment and disclosure of AI assistance}
The author acknowledges the use of OpenAI's ChatGPT during preparation
of this manuscript for mathematical discussion, assistance with proof
strategies and calculations, identification and checking of references,
and revision of the organization and language.  Responsibility for the
arguments, computations, references, and conclusions rests with the
author.

\end{document}